\documentclass[oneside,english]{amsart}
\usepackage[T1]{fontenc}
\usepackage[latin9]{inputenc}
\usepackage{geometry}
\usepackage{color}
\usepackage{babel}
\usepackage{refstyle}
\usepackage{units}
\usepackage{amsthm}
\usepackage{amssymb}
\usepackage{graphicx}
\usepackage{xargs}[2008/03/08]
\usepackage[unicode=true,pdfusetitle,
 bookmarks=true,bookmarksnumbered=false,bookmarksopen=false,
 breaklinks=false,pdfborder={0 0 0},pdfborderstyle={},backref=false,colorlinks=true]
 {hyperref}
\hypersetup{
 citecolor=blue}

\makeatletter

\AtBeginDocument{\providecommand\subsecref[1]{\ref{subsec:#1}}}
\AtBeginDocument{\providecommand\corref[1]{\ref{cor:#1}}}
\AtBeginDocument{\providecommand\thmref[1]{\ref{thm:#1}}}
\AtBeginDocument{\providecommand\defref[1]{\ref{def:#1}}}
\AtBeginDocument{\providecommand\figref[1]{\ref{fig:#1}}}
\AtBeginDocument{\providecommand\lemref[1]{\ref{lem:#1}}}
\AtBeginDocument{\providecommand\appref[1]{\ref{app:#1}}}
\AtBeginDocument{\providecommand\eqref[1]{\ref{eq:#1}}}
\AtBeginDocument{\providecommand\enuref[1]{\ref{enu:#1}}}
\RS@ifundefined{subsecref}
  {\newref{subsec}{name = \RSsectxt}}
  {}
\RS@ifundefined{thmref}
  {\def\RSthmtxt{theorem~}\newref{thm}{name = \RSthmtxt}}
  {}
\RS@ifundefined{lemref}
  {\def\RSlemtxt{lemma~}\newref{lem}{name = \RSlemtxt}}
  {}

\numberwithin{equation}{section}
\numberwithin{figure}{section}
\theoremstyle{plain}
\newtheorem{thm}{\protect\theoremname}[section]
  \theoremstyle{definition}
  \newtheorem{defn}[thm]{\protect\definitionname}
  \theoremstyle{remark}
  \newtheorem{rem}[thm]{\protect\remarkname}
  \theoremstyle{plain}
  \newtheorem{lem}[thm]{\protect\lemmaname}
  \theoremstyle{plain}
  \newtheorem{cor}[thm]{\protect\corollaryname}

\@ifundefined{date}{}{\date{}}
\hypersetup{pdfstartview=}
\newref{def}{name=Definition~}
\newref{ex}{name=Example~}
\newref{sub}{name=Subsection~}
\newref{cor}{name=Corollary~}
\newref{rem}{name=Remark~}
\newref{prop}{name=Proposition~}
\newref{lem}{name=Lemma~}
\newref{thm}{name=Theorem~}
\newref{sec}{name=Section~}
\newref{app}{name=Appendix~}

\hypersetup{citecolor=blue}

\makeatother

  \providecommand{\corollaryname}{Corollary}
  \providecommand{\definitionname}{Definition}
  \providecommand{\lemmaname}{Lemma}
  \providecommand{\remarkname}{Remark}
\providecommand{\theoremname}{Theorem}

\begin{document}
\global\long\def\al{\alpha}
\global\long\def\be{\beta}
\global\long\def\ga{\gamma}
 \global\long\def\Ga{\Gamma}
 \global\long\def\del{\delta}
 \global\long\def\Del{\Delta}
 \global\long\def\lam{\lambda}
 \global\long\def\Lam{\Lambda}
 \global\long\def\eps{\epsilon}
 \global\long\def\veps{\varepsilon}
 \global\long\def\ka{\kappa}
 \global\long\def\sig{\sigma}
 \global\long\def\Sig{\Sigma}
 \global\long\def\om{\omega}
 \global\long\def\Om{\Omega}
 \global\long\def\vphi{\varphi}

\global\long\def\cA{\mathcal{A}}
\global\long\def\cB{\mathcal{B}}
\global\long\def\cC{\mathcal{C}}
 \global\long\def\cD{\mathcal{D}}
 \global\long\def\cE{\mathcal{E}}
 \global\long\def\cF{\mathcal{F}}
 \global\long\def\cG{\mathcal{G}}
 \global\long\def\cH{\mathcal{H}}
 \global\long\def\cI{\mathcal{I}}
 \global\long\def\cJ{\mathcal{J}}
 \global\long\def\cK{\mathcal{K}}
 \global\long\def\cL{\mathcal{L}}
 \global\long\def\cM{\mathcal{M}}
 \global\long\def\cN{\mathcal{N}}
 \global\long\def\cO{\mathcal{O}}
 \global\long\def\cP{\mathcal{P}}
 \global\long\def\cQ{\mathcal{Q}}
 \global\long\def\cR{\mathcal{R}}
 \global\long\def\cS{\mathcal{S}}
 \global\long\def\cT{\mathcal{T}}
\global\long\def\cU{\mathcal{U}}
\global\long\def\cV{\mathcal{V}}
\global\long\def\cW{\mathcal{W}}
 \global\long\def\cX{\mathcal{X}}
 \global\long\def\cY{\mathcal{Y}}
 \global\long\def\cZ{\mathcal{Z}}

\global\long\def\bA{\mathbb{A}}
\global\long\def\bB{\mathbb{B}}
 \global\long\def\bC{\mathbb{C}}
 \global\long\def\bD{\mathbb{D}}
 \global\long\def\bG{\mathbb{G}}
 \global\long\def\bL{\mathbb{L}}
 \global\long\def\bM{\mathbb{M}}
 \global\long\def\bP{\mathbb{P}}
 \global\long\def\bR{\mathbb{R}}
 \global\long\def\bZ{\mathbb{Z}}
 \global\long\def\bQ{\mathbb{Q}}
\global\long\def\bF{\mathbb{F}}
 \global\long\def\bK{\mathbb{K}}
 \global\long\def\bN{\mathbb{N}}
\global\long\def\bH{\mathbb{H}}
 \global\long\def\bT{\mathbb{T}}
 \global\long\def\bS{\mathbb{S}}

\global\long\def\goa{\mathfrak{a}}
\global\long\def\gob{\mathfrak{b}}
 \global\long\def\goc{\mathfrak{c}}
 \global\long\def\gog{\mathfrak{g}}
 \global\long\def\goh{\mathfrak{h}}
 \global\long\def\gol{\mathfrak{l}}
 \global\long\def\gon{\mathfrak{n}}
 \global\long\def\gop{\mathfrak{p}}
 \global\long\def\gos{\mathfrak{s}}
 \global\long\def\got{\mathfrak{t}}
 \global\long\def\gou{\mathfrak{u}}

\global\long\def\operatorname#1{\mathrm{#1}}
\global\long\def\on#1{\mathrm{#1}}

\global\long\def\SL{\mathrm{SL}}
\global\long\def\SO{\mathrm{SO}}
 \global\long\def\PO{\mathrm{PO}}
 \global\long\def\PGL{\mathrm{PGL}}
 \global\long\def\ASL{\mathrm{ASL}}
\global\long\def\GL{\mathrm{GL}}
 \global\long\def\PSL{\mathrm{PSL}}

\global\long\def\defi{\overset{\on{def}}{=}}
\global\long\def\comp{\textrm{{\tiny\ensuremath{\circ}}}}
\global\long\def\lie{\operatorname{Lie}}
\global\long\def\bignorm#1{\left|#1\right|}
\global\long\def\norm#1{\left\Vert #1\right\Vert }
\global\long\def\wt#1{\widetilde{#1}}
\global\long\def\wh#1{\widehat{#1}}
\global\long\def\set#1{\left\{  #1\right\}  }
\global\long\def\pa#1{\left(#1\right)}
\global\long\def\idist#1{\left\langle #1\right\rangle }
\global\long\def\bigav#1{\left|#1\right|}
\global\long\def\av#1{|#1|}
\global\long\def\diag#1{\on{diag}\left(#1\right)}
\global\long\def\mb#1{\mathbf{#1}}
\global\long\def\tb#1{\textbf{#1}}
\global\long\def\mat#1{\left(\begin{array}{c}
 #1\end{array}\right)}
\global\long\def\br#1{\left[#1\right]}
\newcommandx\smallmat[1][usedefault, addprefix=\global, 1=]{\left(#1\right)}
\global\long\def\scrly#1{\mathscr{#1}}
\global\long\def\tran#1{#1^{t}}

\global\long\def\re{\operatorname{Re}}
\global\long\def\im{\operatorname{Im}}
\global\long\def\Mat{\operatorname{Mat}}
\global\long\def\pr{\operatorname{pr}}
\global\long\def\wstar{\overset{\on w^{*}}{\longrightarrow}}
\global\long\def\spa{\on{span}}
\global\long\def\supp{\on{supp}}
\global\long\def\dsim{\Del_{\on{sym}}}

\global\long\def\adorb{\tilde{x}_{0}A_{\bA}}
\global\long\def\madorb{\mu_{\adorb}}

\global\long\def\lla{\longleftarrow}
\global\long\def\la{\leftarrow}
\global\long\def\lra{\longrightarrow}
\global\long\def\ra{\rightarrow}

\global\long\def\ora#1{\overset{#1}{\rightarrow}}
\global\long\def\olra#1{\overset{#1}{\longrightarrow}}
\global\long\def\olla#1{\overset{#1}{\longleftarrow}}
\global\long\def\ulra#1{\underset{#1}{\longrightarrow}}
\global\long\def\ulla#1{\underset{#1}{\longleftarrow}}

\global\long\def\hra{\hookrightarrow}
\global\long\def\lera{\leftrightarrow}
\global\long\def\onto{\xymatrix{\ar@{>>}[r]}
 }

\global\long\def\slra#1{\stackrel{#1}{\lra}}
\global\long\def\sra#1{\stackrel{#1}{\ra}}

\global\long\def\eq#1{ 
\[
{#1}
\]
 }

\global\long\def\eqlabel#1#2{ 
\begin{equation}
{#2}\label{=0000231}
\end{equation}
 }

\global\long\def\eqalign#1{ 
\begin{align*}
{#1}
\end{align*}
 }

\global\long\def\eqalignlabel#1#2{ 
\begin{align}
{#2}\label{=0000231}
\end{align}
 }

\global\long\def\usnote#1{\marginpar{\color{cyan}\tiny[US] #1}}
 \global\long\def\usadd#1{{\color{cyan}[US] #1}}
 \global\long\def\odnote#1{\marginpar{\color{blue}[OD] #1}}
 \global\long\def\odadd#1{{\color{blue}{\tiny[OD]} #1}}


\global\long\def\norm#1{\left\Vert #1\right\Vert }
\global\long\def\AA{\mathbb{A}}
\global\long\def\QQ{\mathbb{Q}}
\global\long\def\PP{\mathbb{P}}
\global\long\def\CC{\mathbb{C}}
\global\long\def\HH{\mathbb{H}}
\global\long\def\ZZ{\mathbb{Z}}
\global\long\def\NN{\mathbb{N}}
\global\long\def\KK{\mathbb{K}}
\global\long\def\RR{\mathbb{R}}
\global\long\def\FF{\mathbb{F}}
\global\long\def\oo{\mathcal{O}}
\global\long\def\aa{\mathcal{A}}
\global\long\def\bb{\mathcal{B}}
\global\long\def\limfi#1#2{{\displaystyle \lim_{#1\to#2}}}
\global\long\def\pp{\mathcal{P}}
\global\long\def\qq{\mathcal{Q}}
\global\long\def\da{\mathrm{da}}
\global\long\def\dt{\mathrm{dt}}
\global\long\def\dg{\mathrm{dg}}
\global\long\def\ds{\mathrm{ds}}
\global\long\def\dm{\mathrm{dm}}
\global\long\def\dmu{\mathrm{d\mu}}
\global\long\def\dx{\mathrm{dx}}
\global\long\def\dy{\mathrm{dy}}
\global\long\def\dz{\mathrm{dz}}
\global\long\def\dnu{\mathrm{d\nu}}
\global\long\def\flr#1{\left\lfloor #1\right\rfloor }
\global\long\def\nuga{\nu_{\mathrm{Gauss}}}
\global\long\def\diag#1{\mathrm{diag}\left(#1\right)}
\global\long\def\bR{\mathbb{R}}
\global\long\def\Ga{\Gamma}
\global\long\def\PGL{\mathrm{PGL}}
\global\long\def\SL{\mathrm{SL}}
\global\long\def\mb#1{\mathrm{#1}}
\global\long\def\wstar{\overset{w^{*}}{\longrightarrow}}
\global\long\def\vphi{\varphi}
\global\long\def\av#1{\left|#1\right|}
\global\long\def\inv#1{\left(\mathbb{Z}/#1\mathbb{Z}\right)^{\times}}
\global\long\def\cH{\mathcal{H}}
\global\long\def\cM{\mathcal{M}}
\global\long\def\bZ{\mathbb{Z}}
\global\long\def\bA{\mathbb{A}}
\global\long\def\bQ{\mathbb{Q}}
\global\long\def\bP{\mathbb{P}}
\global\long\def\eps{\epsilon}
\global\long\def\on#1{\mathrm{#1}}
\global\long\def\nuga{\nu_{\mathrm{Gauss}}}
\global\long\def\set#1{\left\{  #1\right\}  }
\global\long\def\smallmat#1{\begin{smallmatrix}#1\end{smallmatrix}}
\global\long\def\len{\mathrm{len}}

\title{Equidistribution of divergent orbits of the diagonal group in the
space of lattices}

\author{Ofir David and Uri Shapira}

\address{Department of Mathematics, Technion, Haifa, Israel}

\email{ushapira@tx.technion.ac.il}

\address{Department of Mathematics, Hebrew University, Jerusalem, Israel}

\email{ofir.david@mail.huji.ac.il}
\begin{abstract}
We consider divergent orbits of the group of diagonal matrices in
the space of lattices in Euclidean space. We define two natural numerical
invariants of such orbits: The discriminant - an integer - and the
type - an integer vector. We then study the question of the limit
distributional behaviour of these orbits as the discriminant goes
to infinity. Using entropy methods we prove that for divergent orbits
of a specific type, virtually any sequence of orbits equidistribute
as the discriminant goes to infinity. Using measure rigidity for higher
rank diagonal actions we complement this result and show that in dimension
3 or higher only very few of these divergent orbits can spend all
of their life-span in a given compact set before they diverge.
\end{abstract}

\maketitle

\section{Introduction}

One of the main points of interest in homogeneous dynamics is the
study of $A$-orbits in the space of $n$-dimensional lattices $X=X_{n}:=\Gamma\backslash G$
where $G=\SL_{n}\left(\RR\right),\ \Gamma=\SL_{n}\left(\ZZ\right)$
and $A\leq\SL_{n}\left(\RR\right)$ is the subgroup of diagonal matrices
with nonnegative entries. Many questions in this area can be formulated
by defining a natural sequence of probability measures $\mu_{i}$
supported on orbits $x_{i}A$ (or a finite union of such orbits),
and ask whether this sequence converges to some measure $\mu$, and
in particular whether it equidistribute, namely $\mu=\mu_{Haar}$
is the Haar measure on $X$ (which is the unique $G$-invariant probability
measure on $X$).  

In this paper we focus on divergent orbits, namely orbits of the form
$xA,\ x\in X$ where the map $A\to X$ defined by $a\mapsto xa$ is
proper. Fixing a Haar measure on $A$, its push forward $\mu_{xA}$
to $xA$ is a well defined $A$-invariant locally finite measure.
In particular, we observe that there is a unique $A$-invariant locally
finite measure on $xA$ up to a scalar product. We will be interested
in whether $\mu_{x_{i}A}$ (and averages of such measures) equidistribute
for natural sequences of divergent orbits.

In the interest of defining equidistribution of locally finite measures,
we introduce the following topology. Let $Z$ be a locally compact
second countable Hausdorff space and let $\cM(Z),\bP\cM\left(Z\right)$
denote the space of locally finite positive Borel measures on $Z$
and homothety classes of such (non-zero) measures respectively. For
$\mu\in\cM\left(Z\right)$ we let $\left[\mu\right]$ denote its homothety
class. It is straightforward to define a topology on $\bP\cM(Z)$
such that for $\left[\mu_{n}\right],\left[\mu\right]\in\bP\cM\left(Z\right)$
we have the limit $\lim\left[\mu_{n}\right]=\left[\mu\right]$ if
and only if there exist constants $c_{n}$ such that for any compact
set $K\subseteq Z$, $c_{n}\mu_{n}\mid_{K}\wstar\mu\mid_{K}$, or
equivalently for every $f,g\in C_{c}\left(Z\right)$ for which $\int g\dmu\neq0$
we have $\frac{\int f\dmu_{n}}{\int g\dmu_{n}}\to\frac{\int f\dmu}{\int g\dmu}$.
(see \cite{shapira_limiting_nodate}).

We shall say that a sequence $0\neq\mu_{n}\in\cM\left(Z\right)$ \emph{equidistributes}
if $\left[\mu_{n}\right]\to\left[\mu_{Haar}\right]$.

We continue to describe the families of divergent orbits which are
the focus of this paper. Recall first that we have the identification
$X=\SL_{n}\left(\ZZ\right)\backslash\SL_{n}\left(\RR\right)\cong\PGL_{n}\left(\ZZ\right)\backslash\PGL_{n}\left(\RR\right)$,
i.e. we may consider general lattices up to homothety instead of unimodular
lattices. We similarly consider orbits under the group of diagonal
matrices in $\PGL_{n}\left(\RR\right)$. A finite index subgroup group
of $\ZZ^{n}$ will be referred to as an integral lattice. Recall Mahler
compactness criterion which says that a sequence of lattices $x_{i}\in X$
diverges if and only if the length of the shortest nonzero vector
in $x_{i}$ converges to zero. It is then evident that integral lattices
have divergent $A$-orbits because they contain a nonzero vector on
each of the axes. On the other hand, it was shown in \cite{tomanov_closed_2003}
that any divergent orbit contains an integral lattice. Evidently,
each such orbit contains a unique integral lattice $L$ with minimal
covolume. This lattice is characterized as the unique integral lattice
in the orbit satisfying $\pi_{i}\left(L\right)=\ZZ$ for all $i$
where $\pi_{i}$ is the projection on the $i$-th coordinate. We call
such an integral lattice \emph{axis primitive }and its covolume is
called the \emph{discriminant }of the orbit. We note that if $L=\ZZ^{n}g$
for $g\in Mat_{n}\left(\ZZ\right)$, then $L$ is axis primitive if
and only if each column of $g$ is a primitive vector, i.e. $gcd\left(g_{1,i},...,g_{n,i}\right)=1$
for each $i$. We note that axis primitive is a stronger condition
than just primitive which is an integral lattice $L\leq\ZZ^{n}$ such
that there is no $0<c<1$ such that $cL$ is still contained in $\ZZ^{n}$,
or equivalently if $L=\ZZ^{n}g$, then $gcd\left\{ g_{i,j}\;\mid\;1\leq i,j\leq n\right\} =1$.

A more refinement invariant of an integral lattice is its type. Given
an integral lattice $L$, we define its type to be the finite abelian
group $\nicefrac{\ZZ^{n}}{L}$, which by the fundamental theorem of
abelian groups can be characterized by a vector $\left(q_{1},q_{2},...,q_{n}\right)\in\NN^{n}$
where $q_{i}\mid q_{i+1}$ and $\nicefrac{\ZZ^{n}}{L}\cong\nicefrac{\ZZ}{q_{1}\ZZ}\times\cdots\times\nicefrac{\ZZ}{q_{n}\ZZ}$.
We define the \emph{type }of a divergent orbit to be the type of its
unique axis primitive integral lattice. 

We shall focus on orbits of type $\left(1,q,...,q\right)$. Any lattice
$L$ of type $\left(1,q,...,q\right)$ satisfies $q\ZZ^{n}\subseteq L$
and $\nicefrac{L}{q\ZZ^{n}}\cong\nicefrac{\ZZ}{q\ZZ}$, and if in
addition $L$ is axis primitive, then we must have that $L=q\ZZ^{n}+\ZZ\left(1,p_{1},...,p_{n-1}\right)$
where $1\leq p_{i}\leq q$ and $\left(p_{i},q\right)=1$. \textcolor{blue}{}
This leads us to the following definition.
\begin{defn}
\label{def:families}For $q\in\NN$, we identify the sets $\left(\nicefrac{\ZZ}{q\ZZ}\right)^{\times}$
and $\left\{ 1\leq p\leq q\;\mid\;\left(p,q\right)=1\right\} $. With
this identification in mind we set $\Lambda_{q}=\left\{ \Gamma u_{\bar{p}/q}\ \mid\ \bar{p}\in\left(\left(\nicefrac{\ZZ}{q\ZZ}\right)^{\times}\right)^{n-1}\right\} $
where $u_{\bar{x}}=\left(\begin{array}{cc}
1 & \bar{x}\\
0_{\left(n-1\right)\times1} & I_{\left(n-1\right)\times\left(n-1\right)}
\end{array}\right)\in\SL_{n}\left(\RR\right)$ for $\bar{x}\in\RR^{n-1}$, and identify it in the natural way with
the sets $\left(\left(\nicefrac{\ZZ}{q\ZZ}\right)^{\times}\right)^{n-1}$,
$\left\{ 1\leq p\leq q\;\mid\;\left(p,q\right)=1\right\} ^{n-1}$
and $\left\{ u_{\bar{p}/q}\;\mid\;\bar{p}\in\left(\left(\nicefrac{\ZZ}{q\ZZ}\right)^{\times}\right)^{n-1}\right\} $.
\end{defn}
Thus, the orbits of type $\left(1,q,...,q\right)$ are in one to one
correspondence with $\Lambda_{q}$. Note in particular that we have
$\left|\Lambda_{q}\right|=\varphi\left(q\right)^{n-1}$. With these
notations we shall establish the following result which generalizes
the main result in \cite{david_equidistribution_2017}.
\begin{thm}
\label{thm:main_theorem}The following is true:
\begin{enumerate}
\item There exist $\Lambda_{q}'\subseteq\Lambda_{q}$ with $\limfi q{\infty}\frac{\left|\Lambda_{q}'\right|}{\left|\Lambda_{q}\right|}=1$,
s.t. given $x_{q}\in\Lambda_{q}'$ we have that $\left[\mu_{x_{q}A}\right]\to\left[\mu_{Haar}\right]$. 
\item If $0<\alpha\leq1$ and $\Lambda_{q}'\subseteq\Lambda_{q}$ such that
$\left|\Lambda_{q}'\right|\geq\alpha\left|\Lambda_{q}\right|$, then
$\left[{\displaystyle \sum_{x\in\Lambda_{q}'}}\mu_{xA}\right]\to\left[\mu_{Haar}\right]$.
\item We have the convergence $\left[{\displaystyle \sum_{x\in\Lambda_{q}}}\mu_{xA}\right]\to\left[\mu_{Haar}\right]$.
\end{enumerate}
\end{thm}
\begin{rem}
While the implications $(1)\Rightarrow(2)\Rightarrow(3)$ are easy,
we actually first prove $(3)$ and show that $(3)\Rightarrow(2)\Rightarrow(1)$
using the fact that $\mu_{Haar}$ is ergodic (with respect to the
action of non-trivial elements of $A$), hence an extreme point in
the set of $A$-invariant probability measures.
\end{rem}
Recall that by Mahler criterion, compact sets in $X$ are closed sets
$K$ such that $inf\left\{ \norm v\;\mid\;0\neq v\in L,\;L\in K\right\} >0$.
Given a rational lattice $x\in X$, its $A$-orbit diverges since
$x$ contains an element in each axis and at least one of them contains
a very short vector once we apply an element of $A$ of large norm.
This motivates us to truncate the orbit and focus on its part where
all non-zero vectors on the axis are of length $\geq1$. Recall that
$A$ is naturally isomorphic to $\RR_{0}^{n}=\left\{ \bar{t}\in\RR^{n}\mid\sum_{1}^{n}t_{i}=0\right\} $
via the map $\bar{t}\mapsto diag\left(e^{t_{1}},...,e^{t_{n}}\right)$.
Moreover, for $x\in\Lambda_{q}$, identifying $A$ with the divergent
orbit $xA$ under $a\mapsto xa$, it is easily verified (see \subsecref{Axis-lattices}
for details) that the truncated part of the orbit mentioned above
is $\Delta_{full,q}=\ln\left(q\right)\Delta_{full}$ where
\[
\Delta_{full}=\left\{ \bar{t}\in\RR_{0}^{n}\;\mid\;t_{i}\geq0\geq t_{1}\geq-1,\ i\geq2\right\} .
\]
As $a\mapsto xa$ diverges quickly when $a\notin\Delta_{full,q}$,
we can expect that if $\left[\mu_{xA}\right]$ is close $\left[\mu_{Haar}\right]$,
then its restriction to $x\Delta_{full,q}$ must also be close to
$\mu_{Haar}$ up to some normalization (see \corref{from_locally_finite}
for the details). By this argument and \thmref{main_theorem} we expect
that for most $x\in\Lambda_{q}$, the truncated orbit $x\Delta_{full,q}$
would be roughly equidistributed in $X$. In particular, fixing some
compact set $K\subseteq X$ there should be very few $x\in\Lambda_{q}$
such that $x\Delta_{full,q}\subseteq K$. Indeed, using classification
of $A$-invariant measures with positive entropy (see \cite{einsiedler_invariant_2006})
we establish the following result.
\begin{thm}
\label{thm:bounded_orbits_uniform}Let $n\geq3$. Fix a compact set
$K\subseteq X_{n}$ and let $\Lambda_{q,K}=\left\{ x\in\Lambda_{q}\;\mid\;x\cdot\Delta_{full,q}\subseteq K\right\} $.
Then for any $\varepsilon>0$ we have that $\left|\Lambda_{q,K}\right|=O\left(\left|\Lambda_{q}\right|^{\varepsilon}\right)$.
\end{thm}

\subsection{Continued fraction expansion}

In the case of $n=2$, there is a well known connection between the
$A$-orbits in $X_{2}$ and continued fraction expansions (c.f.e)
of numbers in $\left[0,1\right]$. This connection was capitalized
by the authors in \cite{david_equidistribution_2017}. We recall the
interpretation of the results above in the realm of c.f.e in order
to give some more intuition to them.

Recall that the Gauss map $T:(0,1]\to\left[0,1\right]$ is defined
by $T\left(x\right)=\frac{1}{x}-\flr{\frac{1}{x}}$ and the Gauss-Kuzmin
measure $\nuga$ on $\left[0,1\right]$ is defined by $\nuga\left(f\right):=\frac{1}{\ln\left(2\right)}\int_{0}^{1}f\left(t\right)\frac{1}{1+t}\dt$.
It is well known that $\nuga$ is a $T$-invariant ergodic probability
measure. Hence, by the pointwise ergodic theorem, for almost every
$x\in\left[0,1\right]$, the sequence $\frac{1}{N}\sum_{0}^{N-1}\delta_{T^{i}\left(x\right)}$
equidistribute, i.e. it converges to $\nuga$. 

This claim is certainly not true for all $x$. Indeed, if $\frac{p}{q}\in\QQ$
is rational, then it has a finite c.f.e of length $\len\left(\frac{p}{q}\right)\in\NN$
so that $T^{\len\left(p/q\right)}\left(\frac{p}{q}\right)=0$, and
in particular $T^{\len\left(p/q\right)+1}\left(\frac{p}{q}\right)$
is not defined. In this case we may still ask whether $\nu_{p/q}:=\frac{1}{\len\left(p/q\right)}{\displaystyle \sum_{0}^{\len\left(p/q\right)-1}}\delta_{T^{i}\left(p/q\right)}$
converge to $\nuga$ for a well chosen sequence of rationals.

The interpretation of \thmref{main_theorem} in this language is that
there exist $\Lambda_{q}'\subseteq\inv q$ such that $\limfi q{\infty}\frac{\left|\Lambda_{q}'\right|}{\varphi\left(q\right)}=1$
and for any sequence $p_{q}\in\Lambda_{q}'$ we have that $\nu_{p_{q}/q}\to\nuga$
as $q\to\infty$. 

\thmref{bounded_orbits_uniform} is not applicable for $n=2$ since
the classification of measures with positive entropy only holds for
$n>2$. On the other hand, the argument of maximal entropy still holds
and leads to the following result in \cite{david_equidistribution_2017}.
For any fixed $M$, let 
\[
\Lambda_{q,M}=\left\{ p\in\inv q\;\mid\;\frac{p}{q}=\left[0;x_{1},...,x_{len\left(p/q\right)}\right]\;,\;x_{i}\leq M\right\} .
\]
Then there exists some $\varepsilon=\varepsilon\left(M\right)>0$
such that $\left|\Lambda_{q,M}\right|=O\left(\left|\Lambda_{q}\right|^{1-\varepsilon}\right)$.
There is no known lower bound for the size of these sets, and in particular
the claim that these sets are non empty (for some fixed $M$) is known
as Zaremba's conjecture.

\subsection{Equidistribution over the adeles}

Another conceptual viewpoint of these results is achieved by lifting
the discussion to the adeles.

Let $\AA$ denote the ring of adeles over $\QQ$ and consider the
space $X_{\AA}:=\Gamma_{\AA}\backslash G_{\AA}$, where $G_{\AA}=\PGL_{n}\left(\AA\right)$
and $\Gamma_{\AA}=\PGL_{n}\left(\QQ\right)$. Let $A_{\AA}\leq G_{\AA}$
denote the subgroup of diagonal matrices. Note that the orbit $\tilde{x}_{0}A_{\AA}$
is a closed orbit (where $\tilde{x}_{0}$ denotes the identity coset
$\Gamma_{\AA}$). In particular, fixing once and for all a Haar measure
on $A_{\AA}$ we obtain a Haar measure on the quotient $stab_{A_{\AA}}\left(\tilde{x}_{0}\right)\backslash A_{\AA}$
and by pushing the latter into $X_{\AA}$ via the proper embedding
induced by the map $a\mapsto\tilde{x}_{0}a$, we obtain an $A_{\AA}$-invariant
locally finite measure $\mu_{\tilde{x}_{0}A_{\AA}}$ supported on
the closed orbit $\tilde{x}_{0}A_{\AA}$. 

Letting $\bar{u}_{\bar{p}/q}=\left(I,u_{\bar{p}/q},u_{\bar{p}/q},...\right)\in G_{\AA}$
for $\bar{p}\in\Lambda_{q}$, the orbit $\tilde{x}A_{\AA}\bar{u}_{\bar{p}/q}^{-1}$
is mapped under the natural natural epimorphism $\Gamma_{\AA}\backslash G_{\AA}\to\Gamma\backslash G$
to the union of orbits ${\displaystyle \bigcup_{x\in\Lambda_{q}}}xA$.
One can obtain the following result which easily seen to imply \thmref{main_theorem}
and actually these claims are equivalent.
\begin{thm}
\label{thm:adeles}For any choice of $\bar{p}^{(q)}\in\Lambda_{q}$
we have that $\left[\bar{u}_{\bar{p}^{(q)}/q}\mu_{\tilde{x}_{0}A_{\AA}}\right]\to\left[\mu_{\AA,Haar}\right]$
as $q\to\infty$.
\end{thm}
A similar result was deduced in \cite{david_equidistribution_2017}
from the analogue of \thmref{main_theorem} for the $n=2$ case. Since
the proofs are identical we do not repeat the details here.

We believe that \thmref{adeles} should be true for more general sequences
of push-forwards. In particular for the case $n=2$, \thmref{adeles}
is equivalent to the following statement (see \cite{david_equidistribution_2017}):
Given any $g_{i}\in G_{\AA}$ such that (1) the real component of
$g_{i}$ is the identity and (2) the projection to $\nicefrac{G_{\AA}}{A_{\AA}}$
diverges, then we have the convergence in homothety $\left[g_{i}\cdot\mu_{\tilde{x}_{0}A_{\AA}}\right]\to\left[\mu_{\AA,Haar}\right]$.
We believe that some version of this statement should be true in higher
dimension, though there should be further conditions on $g_{i}$,
for example in order for $g_{i}\cdot\mu_{\tilde{x}_{0}A_{\AA}}$ not
to get ``stuck'' in a subspace.

\subsection{Acknowledgment}

The authors would like to thank Elon Lindenstrauss for valuable discussions
and acknowledge the support of ISF 357/13.

\section{\label{subsec:Axis-lattices}Axis lattices and truncated $A$-orbits}

In this paper we will be interested in the orbits of lattices of the
form $\ZZ^{n}u_{\bar{p}/q}$ where $\bar{p}\in\Lambda_{q}$ (see \defref{families}).
These lattices contain a vector on each axis, therefore the map $a\mapsto\ZZ^{n}u_{\bar{p}/q}a$
from $A$ to $X$ is proper, and we say that the orbit diverges. We
begin with a slightly more generalized discussion to understand this
phenomenon. 

For $\bar{t}\in\RR_{0}^{n}$ we denote by $a\left(\bar{t}\right)$
the diagonal matrix $diag\left(e^{t_{1}},...,e^{t_{n}}\right)\in A$
and set $a_{q}\left(\bar{t}\right):=a\left(\ln\left(q\right)\bar{t}\right)=diag\left(q^{t_{1}},...,q^{t_{n}}\right)$.
When there is no ambiguity, we shall identify between $\RR_{0}^{n}$
and the group $A$ of positive diagonal matrices. For the rest of
the paper we fix a Haar measure $\lambda$ on $A$.
\begin{defn}
Let $L\subseteq\RR^{n}$ be a unimodular lattice. 
\begin{itemize}
\item The lattice $L$ is called an \emph{axis lattice} if $L\cap\RR e_{i}\neq0,\;1\leq i\leq n$.
\item For $1\leq i\leq n$ we write $covol\left(L,i\right):=vol\left(\RR e_{i}/\left(L\cap\RR e_{i}\right)\right)$
(which is finite for axis lattices).
\item Given an axis lattice we define
\begin{align*}
A_{L} & =\left\{ \bar{t}\in\RR_{0}^{n}\;\mid\;covol\left(La\left(\bar{t}\right),i\right)\geq1\;\forall i\in\left\{ 1,...,n\right\} \right\} \\
 & =\left\{ \bar{t}\in\RR_{0}^{n}\;\mid\;t_{i}\geq-\ln\left(covol\left(L,i\right)\right)\;\forall i\in\left\{ 1,...,n\right\} \right\} .
\end{align*}
\item We shall write $\tau_{L}=\frac{1}{n}\sum_{1}^{n}\ln\left(covol\left(L,i\right)\right)$.
\end{itemize}
\end{defn}
We note that any axis lattice $x\in X$ contains in its $A$-orbit
an axis lattice $x'$ such that $covol\left(x',i\right)=covol\left(x',j\right)$
for all $i,j$ and $A_{x'}=\left\{ \bar{t}\in\RR_{0}^{n}\;\mid\;t_{i}\geq-\tau_{x}\right\} $.
It then follows that $\frac{\lambda\left(A_{x}\right)}{\tau_{x}^{\left(n-1\right)}}=\frac{\lambda\left(A_{x'}\right)}{\tau_{x}^{\left(n-1\right)}}=\lambda\left(\left\{ \bar{t}\in\RR_{0}^{n}\;\mid\;t_{i}\geq-1\right\} \right)$
is constant.

Recall that for a divergent orbit $xA$ we let $\mu_{xA}$ be the
locally finite $A$-invariant measure on $xA$ (which is the pushforward
of the Haar measure $\lambda$ on $A$). By \cite{tomanov_closed_2003}
the orbit $xA$ contains a rational lattice, so in particular $x$
must be an axis lattice. We denote by $\mu_{x}$ the probability normalization
of the restriction of $\mu_{xA}$ to $x\cdot A_{x}$. 

For a lattice $L$ we set $ht\left(L\right)=\max\left\{ \norm v_{\infty}^{-1}\;\mid\;0\neq v\in L\right\} $
and define $X^{\leq M}=\left\{ x\in X\;\mid\;ht\left(x\right)\leq M\right\} $
(and similarly we define $X^{<M},X^{\geq M}$ and $X^{>M}$). 

The following lemma tells us that when investigating the orbit $xA$,
one may focus only on $xA_{x}$.
\begin{lem}
Let $f\in C_{c}\left(X\right)$ with $supp\left(f\right)\subseteq X^{\leq M}$
and $x\in X$ with $\tau_{x}\geq M$. Then 
\[
\left|\left(\lambda\left(A_{x}\right)^{-1}\mu_{xA}-\mu_{x}\right)\left(f\right)\right|\leq n\cdot2^{n}\frac{M}{\tau_{x}}\norm f.
\]
\end{lem}
\begin{proof}
Up to choosing a different representative of the $A$-orbit, we may
assume that $A_{x}=\left\{ \bar{t}\in\RR_{0}^{n}\;\mid\;\min t_{i}\geq-\tau_{x}\right\} $.
Setting $A_{x,M}=\left(1+\frac{M}{\tau_{x}}\right)A_{x}$ we obtain
that $ht\left(xa\left(\bar{t}\right)\right)>M$ for every $\bar{t}\in\RR_{0}^{n}\backslash A_{x,M}$.
We conclude that
\begin{align*}
\left|\left(\lambda\left(A_{x}\right)^{-1}\mu_{xA}-\mu_{x}\right)\left(f\right)\right| & =\left|\left(\lambda\left(A_{x}\right)^{-1}\mu_{xA}\mid_{x\left(A_{x,M}\cap A_{x}^{c}\right)}\right)\left(f\right)\right|\leq\left(\frac{\lambda\left(A_{x,M}\right)}{\lambda\left(A_{x}\right)}-1\right)\norm f\\
 & =\left[\left(1+\frac{M}{\tau_{x}}\right)^{\left(n-1\right)}-1\right]\norm f\leq\left[\left(n-1\right)\left(1+\frac{M}{\tau_{x}}\right)^{\left(n-2\right)}\frac{M}{\tau_{x}}\right]\norm f\leq n\cdot2^{n}\frac{M}{\tau_{x}}\norm f.
\end{align*}
\end{proof}
From this we conclude the following corollary which says that taking
the limit of the locally finite $A$-invariant measures and their
corresponding restrictions is basically equivalent.
\begin{cor}
\label{cor:from_locally_finite}Let $\mathcal{F}_{i}=\left\{ x_{i,1},...,x_{i,m\left(i\right)}\right\} $
be a family of axis lattices such that $\tau_{i}=\tau_{x_{i,j}}$
is independent of $1\leq j\leq m\left(i\right)$ and $\tau_{i}\to\infty$.
Then $\frac{1}{\left|\mathcal{F}_{i}\right|}\sum_{j}\mu_{x_{i,j}}\wstar\mu$
implies that $\left[\sum_{j}\mu_{x_{i,j}A}\right]\to\left[\mu\right]$.
\end{cor}
This corollary allow us to prove \thmref{main_theorem} by proving
its analogue with the probability measures $\mu_{x}$ instead of the
locally finite measures $\mu_{xA}$.\\

The measures $\frac{1}{\left|\mathcal{F}_{i}\right|}\sum_{j}\mu_{x_{i,j}}$
in the corollary above are averages over the elements in $\mathcal{F}_{i}$,
and each $\mu_{x}$ is an average (integral) over part of the orbit.
To study these notions we use the following notations.
\begin{defn}
Let $\mu$ be a probability measure on $X$, $\nu$ a probability
measure on $\RR_{0}^{n}$ and let $\Lambda\subseteq X$ be a finite
subset.
\begin{enumerate}
\item We define $\delta_{\Lambda}=\frac{1}{\left|\Lambda\right|}\sum_{x\in\Lambda}\delta_{x}$
where $\delta_{x}$ is the Dirac measure at $x$.
\item For $\Omega\subseteq\RR_{0}^{n}$ with positive Lebesgue measure define
$\mu^{\Omega}=\frac{1}{\lambda\left(\Omega\right)}\int_{\Omega}\left(a\left(-\bar{t}\right)\mu\right)\mathrm{d\lambda}(\bar{t})$.
\end{enumerate}
\end{defn}
Using these notations, if $x$ is an axis lattice, then $\mu_{x}=\delta_{x}^{A_{x}}$.
If $x\in\Lambda_{q}$, then it is easily verified that $A_{x}=\ln\left(q\right)\Delta_{full}$,
so in particular $A_{x}$ is only a function of $q$. It follows that
$\frac{1}{\left|\Lambda_{q}\right|}{\displaystyle \sum_{x\in\Lambda_{q}}}\mu_{x}=\delta_{\Lambda_{q}}^{\ln\left(q\right)\Delta_{full}}$,
so in order to prove \thmref{main_theorem} part (3), it is enough
to prove that $\delta_{\Lambda_{q}}^{\ln\left(q\right)\Delta_{full}}\to\mu_{Haar}$.

The permutation group $S_{n}$ acts naturally on lattices in $X_{n}$
and on elements from $A\cong\RR_{0}^{n}$. Moreover, for $x\in X_{n},\;a\in A$
and $\sigma\in S_{n}$ we have that $\sigma\left(xa\right)=\sigma\left(x\right)\sigma\left(a\right)$.
 In the next lemma we show that this action restricts to an action
on $\Lambda_{q}\Delta_{full,q}$ . In particular, for $\sigma=\left(1,...,n\right)$
the fundamental domain is $\Lambda_{q}\Delta_{q}$ where $\Delta_{q}=\ln\left(q\right)\Delta$
and 
\begin{align*}
\Delta & =\left\{ \bar{t}\in\RR_{0}^{d+1}\;\mid\;t_{i}\geq0\geq t_{1}\geq-1+t_{i},\ \forall i\geq2\right\} ,
\end{align*}
which allows us to consider in our computations the measure $\delta_{\Lambda_{q}}^{\ln\left(q\right)\Delta}$
instead of $\delta_{\Lambda_{q}}^{\ln\left(q\right)\Delta_{full}}$
(see \figref{example2}). 
\begin{lem}
\label{lem:fundamental}For any $q$ we have that $\delta_{\Lambda_{q}}^{\ln\left(q\right)\Delta_{full}}=\frac{1}{n}{\displaystyle \sum_{i=1}^{n}}\sigma_{n}^{i}\left(\delta_{\Lambda_{q}}^{\ln\left(q\right)\Delta}\right)$
where $\sigma_{n}=\left(1,...,n\right)\in S_{n}$.
\end{lem}
\begin{proof}
In order to make the problem more symmetric we consider the set $\tilde{\Lambda}_{q}=\Lambda_{q}a_{q}\left(\frac{1}{n}\left(1-n,1,...,1\right)\right)$.
Thus the lattices $\tilde{x}\in\tilde{\Lambda}_{q}$ are exactly those
with basis of the form $\left\{ q^{1/n}e_{2},...,q^{1/n}e_{n},q^{1/n}\cdot\frac{1}{q}\left(e_{1}+\sum_{j>1}p_{j}e_{j}\right)\right\} $
where $p_{j}\in\inv q$. Equivalently these are lattices generated
by $q^{1/n}\ZZ^{n}$ and a vector $q^{\frac{1-n}{n}}\sum_{1}^{n}p_{i}e_{i}$
where $p_{i}\in\inv q$. It follows that for $\tilde{x}\in\tilde{\Lambda}_{q}$
we have that $A_{\tilde{x}}=\tilde{\Delta}_{full,q}=\ln\left(q\right)\tilde{\Delta}_{full}$
where $\tilde{\Delta}_{full}=\left\{ \bar{t}\in\RR_{0}^{n}\;\mid\;t_{i}\geq-\frac{1}{n}\right\} $. 

Clearly, the sets $\tilde{\Lambda}_{q}$ and $\tilde{\Delta}_{full}$
is stable under $S_{n}$ and hence also $\tilde{\Lambda}_{q}\tilde{\Delta}_{full,q}=\Lambda_{q}\Delta_{full,q}$.
If $H\leq S_{n}$ and $\tilde{F}$ is a fundamental domain of $H$
in $\tilde{\Delta}_{full}$, then $\tilde{\Lambda}_{q}F$ is a fundamental
of $H$ in $\tilde{\Lambda}_{q}\tilde{\Delta}_{full,q}$. Using the
fact that $H$ permutes the elements in $\tilde{\Lambda}_{q}$ we
obtain that therefore 
\[
\delta_{\Lambda_{q}}^{\ln\left(q\right)\Delta_{full}}=\delta_{\tilde{\Lambda}_{q}}^{\ln\left(q\right)\tilde{\Delta}_{full}}=\frac{1}{\left|H\right|}{\displaystyle \sum_{h\in H}}\delta_{\tilde{\Lambda}_{q}}^{\ln\left(q\right)h\left(F\right)}=\frac{1}{\left|H\right|}{\displaystyle \sum_{h\in H}}h\left(\delta_{\tilde{\Lambda}_{q}}^{\ln\left(q\right)F}\right).
\]
Carrying this argument back to $\Lambda_{q}$ we need to take the
fundamental domain $F=\tilde{F}+\frac{1}{n}\left(1-n,1,...,1\right)$
instead and obtain $\delta_{\Lambda_{q}}^{\ln\left(q\right)\Delta_{full}}=\frac{1}{\left|H\right|}{\displaystyle \sum_{h\in H}}h\left(\delta_{\Lambda_{q}}^{\ln\left(q\right)F}\right)$.
If $H=\left\langle \left(1,...,n\right)\right\rangle $, then we can
choose $\tilde{F}=\left\{ \bar{t}\in\RR_{0}^{n}\;\mid\;t_{1}\geq t_{i}\geq-\frac{1}{n},\;\forall i\right\} $
and then $F=\Delta$, thus completing the proof.  
\end{proof}
\begin{cor}
\label{cor:fund_domain}If $\delta_{\Lambda_{q}}^{\ln\left(q\right)\Delta}\wstar\mu_{Haar}$,
then $\delta_{\Lambda_{q}}^{\ln\left(q\right)\Delta_{full}}\wstar\mu_{Haar}$.
\end{cor}
\begin{proof}
Let $\sigma_{n}=\left(1,...,n\right)\in S_{n}$ and set $P_{\sigma_{n}}=\sum E_{\sigma_{n}\left(i\right),i}$
be the corresponding permutation matrix (where $E_{i,j}$ is the zero
matrix with $1$ in the $\left(i,j\right)$ coordinate). Let $g=P_{\sigma_{n}}$
if $n$ is odd and $g=diag\left(-1,1,...,1\right)P_{\sigma_{n}}$
if $n$ is even, so in any way we have that $g\in\SL_{n}\left(\RR\right)$
and $g\left(x\right)=P_{\sigma_{n}}\left(x\right)=h\left(x\right)$
for any $x\in X_{n}$. 

By \lemref{fundamental}, we have that $\delta_{\Lambda_{q}}^{\ln\left(q\right)\Delta_{full}}=\frac{1}{n}{\displaystyle \sum_{i=1}^{n}}\sigma_{n}^{i}\left(\delta_{\Lambda_{q}}^{\ln\left(q\right)\Delta}\right)$.
By the assumption $\delta_{\Lambda_{q}}^{\ln\left(q\right)\Delta}\wstar\mu_{Haar}$
which is $G$-invariant, hence $\delta_{\Lambda_{q}}^{\ln\left(q\right)\Delta_{full}}=\frac{1}{n}{\displaystyle \sum_{i=1}^{n}}\sigma_{n}^{i}\left(\delta_{\Lambda_{q}}^{\ln\left(q\right)\Delta}\right)=\frac{1}{n}{\displaystyle \sum_{i=1}^{n}}g^{i}\left(\delta_{\Lambda_{q}}^{\ln\left(q\right)\Delta}\right)\wstar\mu_{Haar}$.
\end{proof}
\begin{figure}
\begin{centering}
\includegraphics[scale=0.4]{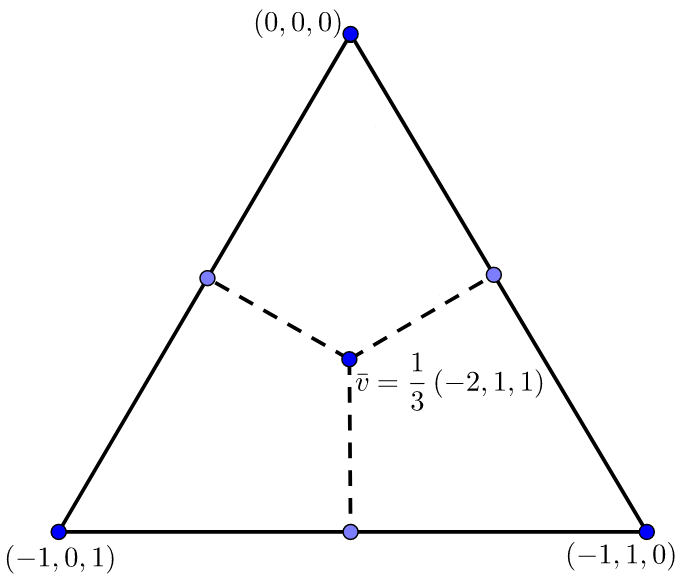}
\par\end{centering}
\caption{For $d=2$, the whole triangle is $\Delta_{full}$, while the upper
third part is $\Delta$. \label{fig:example2}}
\end{figure}

\section{Proofs of the main results}

Once we restrict our attention to $\Lambda_{q}\cdot\Delta$, the flow
induced by the action of the line spanned by $\frac{1}{n}\left(1-n,1,...,1\right)$,
which expands ``uniformly'' the set $\Lambda_{q}$, will play an
important role in our analysis. 
\begin{defn}
\label{def:one_param_measures}We set $\overrightarrow{\mathrm{v}}=\frac{1}{n}\left(1-n,1,...,1\right)$.
Let $E\left(\Delta\right)=\left\{ \left(\bar{w},r\right)\in\Delta\times\RR_{>0}\;\mid\;\bar{w}+r\cdot\overrightarrow{\mathrm{v}}\in\Delta\right\} $,
i.e. the intervals in the direction $\overrightarrow{\mathrm{v}}$
inside $\Delta$. For $\left(\bar{w},r\right)\in E\left(\Delta\right)$
and a probability measure $\mu$ on $X$ we define $\mu^{\left[\bar{w},r,q\right]}=\frac{1}{r}\int_{0}^{r}a_{q}\left(-\left(\bar{w}+r\overrightarrow{\mathrm{v}}\right)\right)\mu\dt$.
\end{defn}
 We note that for $x\in X$, taking $\mu=\delta_{x}$ , the measure
$\delta_{x}^{\left[\bar{w},r,q\right]}$ is the uniform measure on
the $\left\{ xa_{q}\left(\bar{w}+r\overrightarrow{\mathrm{v}}\right)\right\} $
as a part of the orbit $xa_{q}\left(\bar{w}\right)a_{q}\left(\RR\overrightarrow{\mathrm{v}}\right)$.
The main idea of the proof is that given a partial weak star limit
$\mu$ of some sequence $\delta_{\Lambda_{q}'}^{\left[\bar{w},r,q\right]}$,
which is $a_{q}\left(\overrightarrow{\mathrm{v}}\right)$ invariant,
we can give a lower bound on the entropy of $\mu$ by giving a lower
bound for the entropy of $\delta_{\Lambda_{q}'}^{\left[\bar{w},r,q\right]}$
for suitable choices of partitions.

\subsection{\label{subsec:entropy_bound}Entropy lower bound on measure averages}

We begin by recalling the definition of entropy and the uniqueness
of measures of maximal entropy in our setting. For more details the
reader is referred to \cite{einsiedler_entropy_nodate,walters_introduction_2000}.
\begin{defn}
Let $Y$ be any measurable space, $S:Y\to Y$ a measurable function,
$\mu$ a probability measure on $Y$ and $\pp=\left\{ P_{i}\right\} _{1}^{N}$
a finite measurable partition of $Y$. Then
\begin{itemize}
\item The entropy of $\mu$ with respect to $\pp$ is $H_{\mu}\left(\pp\right)=-\sum\mu\left(P_{i}\right)\ln\left(\mu\left(P_{i}\right)\right)$.
\item For $m_{1}<m_{2}$ integers we write $\pp_{m_{1}}^{m_{2}}=\bigvee_{m_{1}}^{m-1}S^{-i}\pp$,
and $\pp^{m}:=\pp_{0}^{m}$.
\item If $\mu$ is $S$-invariant, we set $h_{\mu}\left(S,\pp\right)=\limfi m{\infty}\frac{1}{m}H_{\mu}\left(\pp^{m}\right)$
and $h_{\mu}\left(S\right)=\sup_{\pp}h_{\mu}\left(S,P\right)$.
\end{itemize}
\end{defn}
In the setting of $X_{n}=\SL_{n}\left(\ZZ\right)\backslash\SL_{n}\left(\RR\right)$
we will always use the map $T\left(x\right)=xa_{q}\left(\overrightarrow{v}\right)$,
in which case there is a unique $T$-invariant measure of maximal
entropy. The next theorem is a combination of Theorems 7.6, 7.9 in
\cite{einsiedler_diagonal_nodate} and the fact that $G$ is generated
by $U,U^{tr}$.
\begin{thm}[see \cite{einsiedler_diagonal_nodate}]
 Let $\mu$ be a $T$-invariant probability measure on $X$. Then
$h_{\mu}\left(T\right)\leq h_{\mu_{Haar}}\left(T\right)=d$, and there
is equality if and only if $\mu=\mu_{Haar}$.
\end{thm}
For $\Lambda\subseteq X$ finite and $\left(\bar{w},r\right)\in E\left(\Delta\right)$,
the measure $\delta_{\Lambda}^{\left[\bar{w},r,q\right]}$ is a double
average of measures, once for the points in $\Lambda$ and once in
the direction in which we compute the entropy. The next lemma provides
a lower bound on the entropy for such averages.
\begin{lem}
\label{lem:entropy_lower_bound}Let $Y$ be any measurable space,
let $S:Y\to Y$ be some measurable function and $\pp$ a partition
of $Y$. For a probability measure $\mu$ on $Y$ we write $\mu^{k}=\frac{1}{k}\sum_{i=0}^{k-1}S^{i}\mu$.
Then
\begin{enumerate}
\item If $\left(\Omega,\dx\right)$ is a measurable space, $x\mapsto\nu_{x}$
is a measurable map from $\Omega$ to the probability measures on
$Y$ and $\nu=\int_{\Omega}\nu_{x}\dx$, then
\[
H_{\nu}\left(\pp\right)\geq\int_{\Omega}H_{\nu_{x}}\left(\pp\right)\dx.
\]
\item For $m\leq k$ we have that 
\begin{align*}
\frac{1}{m}H_{\mu^{k}}\left(\pp^{m}\right) & \geq\frac{1}{k}H_{\mu}\left(\pp^{n}\right)-\frac{m}{k}\ln\left|\pp\right|
\end{align*}
\end{enumerate}
\end{lem}
\begin{proof}
\begin{enumerate}
\item Since the function $\alpha:x\mapsto-x\ln\left(x\right)$ is concave
in $\left[0,1\right]$, we obtain that
\[
H_{\nu}\left(\pp\right)=\sum_{\qq\in\pp}\alpha\left(\nu\left(\qq\right)\right)=\sum_{\qq\in\pp}\alpha\left(\int\nu_{x}\left(\qq\right)\dx\right)\geq\sum_{\qq\in\pp}\int\alpha\left(\nu_{x}\left(\qq\right)\right)\dx=\int H_{\nu_{x}}\left(\pp\right)\dx.
\]
\item Write $k=lm+r\leq m\left(l+1\right)$ where $0\leq r<m$, and using
subadditivity we get that for $0\leq u\leq m-1$ we have
\begin{eqnarray*}
H_{\mu}\left(\pp^{k}\right) & \leq & H_{\mu}\left(\pp^{lm+r}\right)\\
 & \leq & \sum_{i=0}^{u-1}H_{\mu}\left(T^{-i}\pp\right)+\sum_{v=0}^{l-1}H_{\mu}(\pp_{vm+u}^{\left(v+1\right)m+u})+\sum_{i=lm+u}^{lm+m-1}H_{\mu}(T^{-i}\pp)\\
 & \leq & m\log\left|\pp\right|+\sum_{v=0}^{l-1}H_{T^{vm+u}\mu}(\pp^{m}).
\end{eqnarray*}
Summing over $0\leq u\leq m-1$ we get that 
\begin{eqnarray*}
mH_{\mu}\left(\pp^{k}\right)-m^{2}\ln\left|\pp\right| & \leq & \sum_{u=0}^{m-1}\sum_{v=0}^{l-1}H_{\left(T^{vm+u}\mu\right)}(\pp^{m})=\sum_{j=0}^{lm-1}H_{\left(T^{j}\mu\right)}(\pp_{m})\\
 & \leq & \sum_{j=0}^{k-1}H_{\left(T^{j}\mu\right)}(\pp^{m})\leq kH_{\mu^{k}}\left(\pp^{m}\right),
\end{eqnarray*}
where in the last step we used part (1). It follows that $\frac{1}{m}H_{\mu^{k}}\left(\pp^{m}\right)\geq\frac{1}{k}H_{\mu}\left(\pp^{k}\right)-\frac{m}{k}\ln\left|\pp\right|.$
\end{enumerate}
\end{proof}

\subsection{Small balls have small measure}

The second part of \lemref{entropy_lower_bound} should be thought
of as instead of taking the average of the measure $\mu$ along the
$T$-direction, we take a greater refinement of the partition $\pp$.
Thus, the atoms in $\pp^{k}$ become smaller and we would expect them
to have small $\mu$ measure, which would provide a lower bound on
the entropy since
\[
H_{\mu}\left(\pp^{k}\right)=-\sum_{P\in\pp^{k}}\mu\left(P\right)\ln\left(\mu\left(P\right)\right)\geq-\sum_{P\in\pp^{k}}\mu\left(P\right)\left[\max_{Q\in\pp^{k}}\ln\left(\mu\left(Q\right)\right)\right]=-\max_{Q\in\pp^{k}}\ln\left(\mu\left(Q\right)\right).
\]

The main problem in this argument is that the space $X$ is not compact,
hence $\pp$ has a noncompact atom, and therefore not all of the atoms
in $\pp^{k}$ are small. After defining what ``small'' is in our
context, we salvage this argument in \lemref{bowen_control} by showing
that most of the atoms in $\pp^{k}$ are small.
\begin{defn}
\label{def:small_balls}For $\eta>0,N\geq0$ we set 
\begin{align*}
V_{\eta,N} & =\left\{ \sum_{i,j=0}^{d}a_{i,j}E_{i,j}\in Mat_{n}\left(\RR\right)\;\mid\;\left|a_{i,j}\right|\leq\begin{cases}
\eta e^{-N} & i=1,\;j\geq2\\
\eta & i\geq2\;or\;i=j=1
\end{cases}\quad\quad\right\} ,\\
B_{\eta,N} & =\left(I+V_{\eta,N}\right)\cap\SL_{n}\left(\RR\right).
\end{align*}
When $N=0$ we write $V_{\eta}=V_{\eta,0}$ and $B_{\eta}=B_{\eta,0}$.
\end{defn}

\begin{defn}
\label{def:partitions}Recall that for a lattice $L\subseteq\RR^{n}$
we write $ht\left(L\right)=ht_{\infty}\left(L\right)=\max\left\{ \norm v_{\infty}^{-1}\mid0\neq v\in L\right\} $
and set $X^{\geq M}=\left\{ x\in X\ \mid\ ht(x)>M\right\} $. A (finite
measurable) partition $\pp$ of $X$ is called an \emph{$\left(M,\eta\right)$
partition} if $\pp=\left\{ P_{0},P_{1},...,P_{n}\right\} $ where
$P_{0}=X^{\geq M}$ and $P_{i}\subseteq x_{i}B_{\eta}$ , $x_{i}\in\Gamma\backslash G$
for $1\leq i\leq n$. 
\end{defn}
The next lemma originally appeared as Lemma 4.5 in \cite{einsiedler_distribution_2012}
for dimension 2, where a slight inaccuracy was corrected in \cite{david_equidistribution_2017}.
The generalization to high dimension is straight forward, and for the
reader's convenience we add its proof in \appref{bowen}.
\begin{lem}[see \cite{einsiedler_distribution_2012}]
\label{lem:bowen_control}For any $M>1$ there exists some $0<\eta_{0}=\eta_{0}\left(M,n\right)$
which satisfies the following. For any $0<\eta\leq\eta_{0}$, an $\left(M,\frac{1}{4}\eta\right)$
partition $\pp$ of $X$ and $\kappa\in\left(0,1\right)$, $N>0$,
there exists some $X'\subseteq X^{\leq M}$ such that 
\begin{enumerate}
\item $X'$ is a union of $S_{1},...,S_{l}\in\pp^{N}$, each of which is
contained in a union of at most $C^{\kappa N}$ many balls of the
form $gB_{\eta,N}$ with $g\in S_{j}$ for some absolute constant
$C$.
\item $\mu\left(X'\right)\geq1-\mu\left(X^{>M}\right)-\mu^{N}\left(X^{>M}\right)\kappa^{-1}$
for every probability measure $\mu$ on $X$.
\end{enumerate}

\end{lem}
Recall that we want to give a lower bound on the entropy where the
idea is to show that most of the atoms in $\pp^{N}$ have small measure.
By \lemref{bowen_control}, most of these atoms are contained in a
union of small balls $zB_{\eta,N}$, hence it is important to show
that these balls have small measure.

Given $\Lambda\subseteq\Lambda_{q}$ and $\left(\bar{w},r\right)\in\Delta$,
the measures $\delta_{\Lambda}^{\left[\bar{w},r,q\right]}$ are averages
in the direction of the flow $T$ and in the ``horospherical'' direction
represented by $\Lambda$. While the $T$-average allows us to talk
about $T$-invariant measure, the upper bound on the measure $\delta_{\Lambda}^{\left[\bar{w},r,q\right]}\left(zB_{\eta,N}\right)$
will arise from the horospherical direction.
\begin{lem}
\label{lem:separation} There exists a constant $C\left(n\right)\geq1$
such that the following is true. Given $x\in X$, $0<\eta<\min\left\{ \frac{1}{4n},\frac{1}{ht\left(x\right)}\right\} $
and $q\in\NN$, for any $\bar{t}\in\Delta$ and $0\leq N\leq\ln\left(q\right)\left(1+t_{0}-\left({\displaystyle \max_{i\geq1}}t_{i}\right)\right)$
we have that $\left|\Lambda_{q}a_{q}\left(\bar{t}\right)\cap xB_{\eta,N}\right|\leq C\left(n\right)\left(\frac{q}{e^{N}}\right)^{n-1}$.
\end{lem}
\begin{proof}
Writing $x=\Gamma g$ where $g\in\SL_{n}\left(\RR\right)$, an upper
bound for $\left|\Lambda_{q}a_{q}\left(\bar{t}\right)\cap xB_{\eta,N}\right|$
is given by the number of pairs $\left(\gamma,\bar{p}\right)\in\SL_{n}\left(\ZZ\right)\times\Lambda_{q}$
such that 
\[
u_{-\bar{p}/q}\gamma g\in a_{q}\left(\bar{t}\right)B_{\eta,N}^{-1}\subseteq a_{q}\left(\bar{t}\right)B_{2\eta,N},
\]
where $B_{\eta,N}^{-1}\subseteq B_{2\eta,N}$ by \lemref{inverse_ball}. 

Letting $\gamma_{i}$ be the rows of $\gamma$, considering the $i$'th
row of the inclusion above we obtain that 
\begin{align}
\gamma_{0}z-\sum_{1}^{d}\frac{p_{i}}{q}\gamma_{i}g & \in q^{t_{0}}\left(e_{0}+2\eta\left(-1,1\right)\times2\eta\left(-e^{-N},e^{-N}\right)^{n-1}\right)\label{eq:first_row}\\
\gamma_{i}g & =q^{t_{i}}\left(e_{i}+v_{i}\right),\;v_{i}\in2\eta\left(-1,1\right)^{n},\quad i=2,...,n.\label{eq:rest_rows}
\end{align}
Suppose that $\gamma_{0}'\in\ZZ^{n}$ such that $\gamma_{1}',\gamma_{2},...,\gamma_{n}$
form a matrix in $\SL_{n}\left(\ZZ\right)$. Writing $\gamma_{0}'=\sum_{1}^{n}k_{i}\gamma_{i},\ k_{i}\in\ZZ$,
we get that $1=\det\left(\gamma_{1}'\mid\gamma_{2}\mid\cdots\mid\gamma_{n}\right)=\det\left(k_{1}\gamma_{1}\mid\gamma_{2}\mid\cdots\mid\gamma_{n}\right)=k_{1}$.
Hence, once the $\gamma_{i},\;i=2,...,n$ are chosen, any additional
solution for \eqref{first_row} has the form $\gamma_{1}g+\sum_{2}^{n}\left(k_{i}-\frac{p_{i}}{q}\right)\gamma_{i}g$
where $\gamma_{1}$ is also fixed. Denoting by $\pi:\RR^{n}\to\RR^{n-1}$
the projection on the hyperplane corresponding to the last $n-1$
coordinate and applying it to \eqref{first_row} we obtain
\begin{equation}
\sum_{1}^{d}m_{i}q^{t_{i}}\left(e_{i}+\pi\left(v_{i}\right)\right)\in-q\pi\left(\gamma_{0}g\right)+2\eta q^{t_{0}}q\left(-e^{-N},e^{-N}\right)^{n-1},\ \ m_{i}=k_{i}q-p_{i}.\label{eq:first_row_II}
\end{equation}

To summarize, we wish to bound from above the number of $\gamma_{i}\in\ZZ^{n},\;m_{i}\in\ZZ$
for $i=2,...,n$ which satisfy \eqref{rest_rows} and \eqref{first_row_II}.
Equation (\ref{eq:rest_rows}) merely asks how many lattice points
$\ZZ^{n}g$ has in $q^{t_{i}}\left(e_{i}+2\eta\left(-1,1\right)^{n}\right)$.
Letting $h=ht_{\infty}\left(\ZZ^{d+1}g\right)$, we obtain that any
set $u+\left[0,\frac{1}{3h}\right]^{d+1}$ contains at most one point
from $\ZZ^{n}g$, hence $q^{t_{i}}\left(e_{i}+2\eta\left(-1,1\right)^{n}\right)$
contains at most $\left\lceil 12h\eta q^{t_{i}}\right\rceil ^{n}\leq\left\lceil 12q^{t_{i}}\right\rceil ^{n}\leq\left(2^{4}q^{t_{i}}\right)^{n}$
such points. For \eqref{first_row_II}, we first multiply by the diagonal
matrix $a=diag\left(q^{-t_{2}},...,q^{-t_{n}}\right)$ to reduce to
the problem of counting the lattice points of $L=span_{\ZZ}\left\{ e_{i}+\pi\left(v_{i}\right)q^{t_{i}}a\right\} $
in $-q\pi\left(\gamma_{0}z\right)a+2\eta q^{t_{1}}\frac{q}{e^{N}}\prod_{2}^{n}\left(-q^{-t_{i}},q^{-t_{i}}\right)$.
The lattice $L$ is $\ZZ^{d}\left(I+a^{-1}g_{\eta}a\right)$ where
$,g_{\eta}\in Mat_{n}\left(\RR\right),\norm{g_{\eta}}_{\infty}\leq2\eta<\frac{1}{2n}$.
We shall see in \lemref{ht_conjugation} that it implies that $h'=ht_{\infty}\left(L\right)\leq2$.
It follows that \eqref{first_row_II} has at most $\prod_{2}^{n}\left\lceil 12h'\eta q^{t_{1}-t_{i}}\frac{q}{e^{N}}\right\rceil \leq\prod_{2}^{n}\left(2^{6}q^{t_{1}-t_{i}}\frac{q}{e^{N}}\right)$
solutions, where we used the fact that $q^{t_{1}-t_{i}}\frac{q}{e^{N}}\geq1$.
Thus, the total number of solutions is bounded by
\[
\prod_{2}^{n}\left(2^{4}q^{t_{i}}\right)^{n}\left(2^{6}q^{t_{1}-t_{i}}\frac{q}{e^{N}}\right)=C\left(n\right)\left(\frac{q}{e^{N}}\right)^{n-1}.
\]
\end{proof}
\begin{lem}
\label{lem:ht_conjugation}Let $g\in Mat_{d}\left(\RR\right)$ with
$\norm g_{\infty}<\frac{1}{2d}$ and $t_{1},...,t_{d}\in\RR$. Then
$ht_{\infty}\left(\ZZ^{d}\cdot\left(I+a\left(\bar{t}\right)ga\left(-\bar{t}\right)\right)\right)\leq2$.
\end{lem}
\begin{proof}
We note first that $\norm{g^{m}}_{\infty}\leq d^{m-1}\norm g_{\infty}^{m}\leq\frac{1}{d}\left(\frac{d}{2d}\right)^{m}$,
so that $\left(I+g\right)^{-1}=\sum_{0}^{\infty}\left(-g\right)^{m}$
is well defined. It follows that $a\left(\bar{t}\right)\left(I+g\right)a\left(-\bar{t}\right)$
is invertible, hence $\ZZ^{d}\cdot\left(I+a\left(\bar{t}\right)ga\left(-\bar{t}\right)\right)$
is indeed a lattice in $\RR^{d}$.

Let $\bar{k}=\left(k_{1},...,k_{d}\right)\in\ZZ^{d}$ such that $\norm{\bar{k}\left(I+a\left(\bar{t}\right)ga\left(-\bar{t}\right)\right)}_{\infty}<\frac{1}{2}$.
Computing the $i$-th coordinate produces 
\[
\frac{1}{2}>\left|k_{i}+\sum_{1}^{d}k_{j}e^{t_{j}-t_{i}}g_{j,i}\right|\geq\left|k_{i}\right|-\frac{1}{2d}\sum_{1}^{d}\left|k_{j}\right|e^{t_{j}-t_{i}}.
\]
If $k_{i}\neq0$, i.e. $\left|k_{i}\right|\geq1$, then $\left|k_{i}\right|\leq2\left(\left|k_{i}\right|-\frac{1}{2}\right)<\frac{1}{d}{\displaystyle \sum_{j=1}^{d}}\left|k_{j}\right|e^{t_{j}-t_{i}}$.
Thus there must be some $j=j\left(i\right)$ such that $\left|k_{i}\right|<\left|k_{j}\right|e^{t_{j}-t_{i}}$.
Consider the directed graph on $\left\{ i\mid\left|k_{i}\right|\neq0\right\} $,
with an edge $i\to j$ if $\left|k_{i}\right|<\left|k_{j}\right|e^{t_{j}-t_{i}}$.
As we have just seen, each vertex has an outgoing edge, so the graph
must contain a directed cycle. Moreover, the defining condition of
the edges imply that the graph is transitive, and hence it must has
a self loop at some vertex $i$. This is a contradiction since otherwise
$\left|k_{i}\right|<\left|k_{i}\right|e^{t_{i}-t_{i}}=\left|k_{i}\right|$.
We conclude that $\bar{k}=\bar{0}$, or equivalently $ht_{\infty}\left(\ZZ^{d}\cdot\left(I+a\left(\bar{t}\right)ga\left(-\bar{t}\right)\right)\right)\leq2$.
\end{proof}

\subsection{\label{subsec:Proof_of_bounded_orbits}Proof of \thmref{bounded_orbits_uniform}}

We now combine all of the results so far to provide a lower bound
on the entropy of partial weak star limits $\delta_{\Lambda}^{\left[\bar{w},r,q\right]}$
(which are $a_{q}\left(\overrightarrow{v}\right)$-invariant where
$\overrightarrow{v}=\frac{1}{n}\left(1-n,1,...,1\right)$). Recall
that if a probability measure $\mu$ is a weak star limit of a sequence
of probability measures $\mu_{i}$, then for any $\varepsilon>0$
there is some $M>0$ such that $\mu_{i}\left(X^{>M}\right)<\varepsilon$
for all $i$ big enough. We say that the sequence $\mu_{i}$ \emph{doesn't
exhibit escape of mass }if every partial weak limit is a probability
measure. The next lemma shows that if $\Lambda\subseteq\Lambda_{q}$
is big, and there is no escape of mass, then we get a good lower bound
on the entropy.
\begin{lem}
\label{lem:main_lemma}Let $\pp$ be any $\left(M,\eta\right)$-partition
with $\eta$ small enough (as a function of $M,n$). Let $\Lambda\subseteq\Lambda_{q}$,
$\left(\bar{w},r\right)\in E\left(\Delta\right)$ such that $q^{\left(n-1\right)\left(1-r\right)}\leq\left|\Lambda\right|$
and $r\ln\left(q\right)>1$. Then
\[
\frac{1}{m}H_{\delta_{\Lambda}^{\left[\bar{w},r,q\right]}}\left(\pp_{m}\right)\geq\left[\left(n-1\right)-\frac{1}{r}\left(\left(n-1\right)-\frac{\ln\left|\Lambda\right|}{\ln\left(q\right)}\right)\right]-O\left(\sqrt{\delta_{\Lambda}^{\left[\bar{w},r,q\right]}\left(X^{>M}\right)}+\delta_{\Lambda a_{q}\left(\bar{w}\right)}\left(X^{>M/e}\right)+\frac{m\ln\left|\pp\right|}{r\ln\left(q\right)}\right).
\]
\end{lem}
\begin{proof}
Assume that $\eta$ is small enough and satisfies the conditions of
\lemref{bowen_control} and \lemref{separation} for all $z\in X^{\leq M}$.

First we move from the continuous measure $\delta_{\Lambda}^{\left[\bar{w},r,q\right]}$
to its discrete version defined by $\delta_{\Lambda}^{\left\{ \bar{v},r,q\right\} }:=\frac{1}{\flr{r\ln\left(q\right)}}\sum_{k=0}^{\flr{r\ln\left(q\right)}-1}\delta_{\Lambda a_{q}\left(\bar{v}\right)a\left(k\overrightarrow{v}\right)}$
. We have

\begin{align*}
\delta_{\Lambda}^{\left[\bar{w},r,q\right]} & =\frac{\flr{\ln\left(q\right)r}}{\ln\left(q\right)r}\cdot\left[\int_{0}^{1}\delta_{\Lambda a\left(x\overrightarrow{v}\right)}^{\left\{ \bar{w},r,q\right\} }\dx\right]+\frac{\ln\left(q\right)r-\flr{\ln\left(q\right)r}}{\ln\left(q\right)r}\cdot\left[\frac{1}{\ln\left(q\right)r-\flr{\ln\left(q\right)r}}\int_{\flr{\ln\left(q\right)r}}^{\ln\left(q\right)r}\delta_{\Lambda a_{q}\left(\bar{w}\right)a\left(x\overrightarrow{v}\right)}\dx\right].
\end{align*}
Applying \lemref{entropy_lower_bound} we get that

\begin{align}
\frac{1}{m}H_{\delta_{\Lambda}^{\left[\bar{w},r,q\right]}}\left(\pp^{m}\right) & \geq\frac{1}{m}\left(1-\frac{1}{\ln\left(q\right)r}\right)\int_{0}^{1}H_{\delta_{\Lambda a\left(x\overrightarrow{v}\right)}^{\left\{ \bar{w},r,q\right\} }}\left(\pp^{m}\right)\dx\label{eq:2142}\\
 & \geq\left(1-\frac{1}{\ln\left(q\right)r}\right)\left[\frac{1}{\flr{r\ln\left(q\right)}}\int_{0}^{1}H_{\delta_{\Lambda a_{q}\left(\bar{w}\right)a\left(x\overrightarrow{v}\right)}}\left(\pp^{\flr{r\ln\left(q\right)}}\right)\dx-\frac{m\ln\left|\pp\right|}{\flr{r\ln\left(q\right)}}\right].\nonumber 
\end{align}

Let $X'=X'_{N,\kappa,M}$ be the set defined in \lemref{bowen_control}
where $N=\flr{r\ln\left(q\right)}$. If $P\in\pp^{N}$ and $P\subseteq X'$,
then $P\subseteq\bigcup_{1}^{C_{1}^{\kappa N}}g_{i}B_{\eta,N}$ with
$g_{i}\in X'\subseteq X^{\leq M}$, and hence by \lemref{separation}
we have that $\left|\Lambda_{q}a\left(\bar{t}\right)\cap\Gamma g_{i}B_{\eta,N}\right|\leq C_{2}e^{\left(n-1\right)\left(\ln\left(q\right)-\flr{r\ln\left(q\right)}\right)}$
for any $\bar{t}\in\ln\left(q\right)\Delta$ , where $C_{1},C_{2}$
depend only on the dimension $n$. Set $C=\max\left\{ 1,C_{1},C_{2}e^{n-1}\right\} $,
and note that $C_{2}e^{\left(n-1\right)\left(\ln\left(q\right)-\flr{r\ln\left(q\right)}\right)}\leq Cq^{\left(n-1\right)\left(1-r\right)}$.

We now have that
\begin{align*}
H_{\delta_{\Lambda a\left(\bar{t}\right)}}\left(\pp^{N}\right) & =-\sum_{P\in\pp_{N}}\delta_{\Lambda a\left(\bar{t}\right)}\left(P\right)\ln\left(\delta_{\Lambda a\left(\bar{t}\right)}\left(P\right)\right)\geq-\sum_{P\in\pp_{N}\cap X'}\delta_{\Lambda a\left(\bar{t}\right)}\left(P\right)\ln\left(\delta_{\Lambda a\left(\bar{t}\right)}\left(P\right)\right)\\
 & \geq-\sum_{P\in\pp_{N}\cap X'}\delta_{\Lambda a\left(\bar{t}\right)}\left(P\right)\ln\left(C^{\kappa N+1}\frac{q^{\left(n-1\right)\left(1-r\right)}}{\left|\Lambda\right|}\right)=-\delta_{\Lambda a\left(\bar{t}\right)}\left(X'\right)\ln\left(C^{\kappa N+1}\frac{q^{\left(n-1\right)\left(1-r\right)}}{\left|\Lambda\right|}\right)\\
 & \geq-\left(\kappa N+1\right)\ln\left(C\right)+\delta_{\Lambda a\left(\bar{t}\right)}\left(X'\right)\ln\left(\frac{\left|\Lambda\right|}{q^{\left(n-1\right)\left(1-r\right)}}\right),
\end{align*}
and dividing by $N\leq r\ln\left(q\right)$ we obtain
\begin{equation}
\frac{1}{N}H_{\delta_{\Lambda a\left(\bar{t}\right)}}\left(\pp^{N}\right)\geq\delta_{\Lambda a\left(\bar{t}\right)}\left(X'\right)\left[\left(n-1\right)-\frac{1}{r}\left(\left(n-1\right)-\frac{\ln\left|\Lambda\right|}{\ln\left(q\right)}\right)\right]-\left(\kappa+\frac{1}{N}\right)\ln\left(C\right).\label{eq:2144}
\end{equation}

Note that $\left(n-1\right)\left(1-r\right)\leq\frac{\ln\left|\Lambda\right|}{\ln\left(q\right)}\leq n-1$
implies that $0\leq\left[\left(n-1\right)-\frac{1}{r}\left(\left(n-1\right)-\frac{\ln\left|\Lambda\right|}{\ln\left(q\right)}\right)\right]\leq n-1$. 

Recall from \lemref{bowen_control} that $\delta_{\Lambda a\left(\bar{t}\right)}\left(X'\right)\geq1-\delta_{\Lambda a\left(\bar{t}\right)}\left(X^{>M}\right)-\delta_{\Lambda a\left(\bar{t}\right)}^{N}\left(X^{>M}\right)\kappa^{-1}$.
Applying this inequality with $a\left(\bar{t}\right)=a_{q}\left(\bar{w}\right)a\left(x\overrightarrow{v}\right),\ 0\leq x\leq1$
and integrating over $x$ we obtain that
\begin{align*}
\int_{0}^{1}\delta_{\Lambda a_{q}\left(\bar{w}\right)a\left(x\overrightarrow{v}\right)}^{N}\left(X^{>M}\right)\dx & =\frac{1}{\flr{r\ln\left(q\right)}}\int_{0}^{\flr{r\ln\left(q\right)}}\delta_{\Lambda a_{q}\left(\bar{w}\right)a\left(x\overrightarrow{v}\right)}\left(X^{>M}\right)\dx\leq\delta_{\Lambda}^{\left[\bar{w},r,q\right]}\left(X^{>M}\right)+\frac{2}{r\ln\left(q\right)},
\end{align*}
and additionally we have that $\delta_{\Lambda a_{q}\left(\bar{w}\right)a\left(x\overrightarrow{v}\right)}\left(X^{>M}\right)<\delta_{\Lambda a_{q}\left(\bar{w}\right)}\left(X^{>M/e}\right)$,
so
\begin{equation}
\int_{0}^{1}\delta_{\Lambda a_{q}\left(\bar{w}\right)a\left(x\overrightarrow{v}\right)}\left(X'\right)\dx\geq1-\delta_{\Lambda a_{q}\left(\bar{w}\right)}\left(X^{>M/e}\right)-\delta_{\Lambda}^{\left[\bar{w},r,q\right]}\left(X^{>M}\right)\kappa^{-1}-\frac{2}{r\ln\left(q\right)}.\label{eq:2145}
\end{equation}
Combining \eqref{2142} \eqref{2144} and \eqref{2145} we obtain
that 
\[
\frac{1}{m}H_{\delta_{\Lambda}^{\left[\bar{w},r,q\right]}}\left(\pp^{m}\right)\geq\left[\left(n-1\right)-\frac{1}{r}\left(\left(n-1\right)-\frac{\ln\left|\Lambda\right|}{\ln\left(q\right)}\right)\right]-O\left(\delta_{\Lambda}^{\left[\bar{w},r,q\right]}\left(X^{>M}\right)\kappa^{-1}+\delta_{\Lambda a_{q}\left(\bar{w}\right)}\left(X^{>eM}\right)+\frac{m\ln\left|\pp\right|}{r\ln\left(q\right)}+\kappa\right).
\]
Setting $\kappa=\sqrt{\delta_{\Lambda}^{\left[\bar{w},r,q\right]}\left(X^{>M}\right)}$
if it is nonzero and otherwise $\kappa=\frac{1}{r\ln\left(q\right)}$
proves the lemma.
\end{proof}
Next we generalize the lemma above to averages over larger sets than
just for interval in $\Delta_{q}$ defined by $\left(\bar{w},r\right)$. 
\begin{defn}
For $s>0$ we set $\Delta_{s}=\left\{ \bar{t}\in\Delta\;\mid\;\left|t_{i}-t_{j}\right|\leq s\;\forall i,j\geq2\right\} $. 
\end{defn}
The set $\Delta_{s}$ should be thought of as a tube neighborhood
of the line corresponding to $\left(\bar{0},1\right)\in E\left(\Delta\right)$
(in \figref{example2} it is the line from the top vertex to the middle
of the triangle). Moreover, thinking of $\Delta_{s}$ as a union of
lines $\left(\bar{w},r\right)\in E\left(\Delta\right)$, we can present
$\mu^{\ln\left(q\right)\Delta_{s}}$ as an integral over $\mu^{\left[\bar{w},r,q\right]}$
for any probability measure $\mu$ on $X$. To give a lower bound
on these intervals lengths note that for $\bar{t}\in\Delta_{s}$ we
have that $\bar{t}+r\overrightarrow{v}\in\Delta_{s}$ if and only
if $\bar{t}+r\overrightarrow{v}\in\Delta$. Thus the conditions on
$r$ are that for every $i\ge2$ we have
\[
t_{i}+\frac{r}{n}\geq0\geq-\sum_{2}^{n}\left(t_{j}+\frac{r}{n}\right)=t_{1}+r\frac{\left(1-n\right)}{n}\geq\left(t_{i}+\frac{r}{n}\right)-1.
\]
Equivalently, $r$ is in the interval $\left[-{\displaystyle \min_{i\geq2}}nt_{i},1+t_{1}-{\displaystyle \max_{i\geq2}}t_{i}\right]$.
Shifting $\bar{t}$ to $\bar{t}-\left({\displaystyle \min_{i\geq2}n}t_{i}\right)\overrightarrow{v}$,
we may assume that ${\displaystyle \min_{i\geq2}}t_{i}=0$ and that
the interval length is $1+t_{1}-{\displaystyle \max_{i\geq2}}t_{i}$.
Since $\bar{t}\in\Delta_{s}$, we further get that $t_{i}\leq s$
for all $i\geq2$ and $t_{1}=-\sum_{2}^{n}t_{i}\geq\left(2-n\right)s$,
so the interval's length is at least $1+\left(1-n\right)s$. This
leads to the following definition.
\begin{defn}
Recall that $\lambda$ is a fixed Haar measure on $A\cong\RR_{0}^{n}$.
For $s>0$ define $\nu_{s}$ to be a probability measure on $E\left(\Delta\right)$
such that 
\[
\frac{1}{\lambda\left(\Delta_{s}\right)}\cdot\lambda\mid_{\Delta_{s}}=\int\left[\frac{1}{r}\int_{0}^{r}\delta_{\bar{w}+t\overrightarrow{v}}\dt\right]\dnu_{s}\left(\bar{w},r\right).
\]
More over we can choose such $\nu_{s}$ which satisfies $r\left(\nu_{s}\right):=inf\left\{ r\;\mid\;\left(\bar{w},r\right)\in supp\left(\nu_{s}\right)\right\} =\max\left\{ 0,1+\left(1-n\right)s\right\} $.
\end{defn}
We can now generalize \lemref{main_lemma}.
\begin{thm}
\label{thm:main_pre_theorem}Fix some $s>0$. For any choice of $\Lambda_{q}'\subseteq\Lambda_{q}$,
any weak star limit $\mu$ of $\delta_{\Lambda_{q}'}^{\ln\left(q\right)\Delta_{s}}$
is $A$-invariant. Suppose in addition that 
\begin{enumerate}
\item \label{enu:no_escape}The sequences $\delta_{\Lambda_{q}'}^{\ln\left(q\right)\Delta_{s}}=\int\delta_{\Lambda_{q}'}^{\left[\bar{w},r,q\right]}\dnu_{s}$
and $\int\delta_{\Lambda_{q}'a_{q}\left(\bar{w}\right)}\dnu_{s}$
do not exhibit escape of mass, and
\item \label{enu:length}we have $\ell={\displaystyle \liminf_{q\to\infty}}\frac{\ln\left|\Lambda_{q}'\right|}{\ln\left(q\right)}>0$.
\end{enumerate}
If $\ell=n-1$, then $h_{\mu}\left(T\right)=n-1$. If $\ell<n-1$,
but $1>\left(n-1\right)s$ then $h_{\mu}\left(T\right)\geq\left(n-1\right)\left(1-\frac{1-\ell/\left(n-1\right)}{1-\left(n-1\right)s}\right)$. 

\end{thm}
\begin{proof}
By the assumption on the lack of escape of mass, any weak star partial
limit of $\delta_{\Lambda_{q}'}^{\ln\left(q\right)\Delta_{s}}$ is
a probability measure, and by the structure of $\Delta_{s}$, it is
easily seen to be $A$-invariant. 

In order to apply \lemref{main_lemma}, we need to show that most
of $\delta_{\Lambda}^{\left[\bar{w},r,q\right]},\delta_{\Lambda a_{q}\left(\bar{w}\right)}$
for $\left(\bar{w},r\right)\in supp\left(\nu_{s}\right)$ do not exhibit
escape of mass, and have that $r\left(\nu_{s}\right)>0$. As this
is not true in general,  we restrict to a submeasure.

By condition (\enuref{no_escape}) we can choose $M$ big enough so
that $\delta_{\Lambda_{q}'}^{\ln\left(q\right)\Delta_{s}'}\left(X^{>M}\right),\int\delta_{\Lambda_{q}'a_{q}\left(\bar{w}\right)}\left(X^{>M/e}\right)<\varepsilon$
for all $q$ large enough. Moreover, we can choose $r_{\varepsilon}\geq r\left(\nu_{s}\right)$
with $r_{\varepsilon}>0$ so that 
\[
\Omega=\Omega\left(q,\varepsilon\right)=\left\{ \left(\bar{w},r\right)\mid\delta_{\Lambda_{q}'}^{\left[\bar{w},r,q\right]}\left(X^{>M}\right)\leq\sqrt{\varepsilon},\ \delta_{\Lambda_{q}'a_{q}\left(\bar{w}\right)}\left(X^{>M/e}\right)\leq\sqrt{\varepsilon},\;r\geq r_{\varepsilon}\right\} 
\]
satisfies $\nu_{s}\left(\Omega\right)\geq1-3\sqrt{\varepsilon}$.
We set $\nu_{s,\varepsilon}=\frac{1}{\nu\left(\Omega\right)}\nu_{s}\mid_{\Omega}$
and $\nu_{s,\varepsilon}'=\frac{1}{\nu\left(\Omega^{c}\right)}\nu_{s}\mid_{\Omega^{c}}$
so that $\nu_{s}=\nu_{s}\left(\Omega\right)\cdot\nu_{s,\varepsilon}+\left(1-\nu_{s}\left(\Omega\right)\right)\nu_{s,\varepsilon}'$. 

Choose an $\left(M,\eta\right)$ partition $\pp$ where $\mu\left(\partial P\right)=0$
for all $P\in\pp$ (there are such partitions for arbitrarily large
$M$ and any small enough $\eta>0$). For such a partition we have
that $\frac{1}{m}H_{\mu}\left(\pp^{m}\right)\geq\frac{1}{m}\limsup H_{\delta_{\Lambda_{q}'}^{\ln\left(q\right)\Delta_{s}}}\left(\pp^{m}\right)$
for any $m\in\NN$. Using \lemref{entropy_lower_bound} twice, once
from $\delta_{\Lambda_{q}'}^{\ln\left(q\right)\Delta_{s}}=\int\delta_{\Lambda_{q}'}^{\left[\bar{w},r,q\right]}\dnu_{s}$
to $\delta_{\Lambda_{q}'}^{\nu_{s,\varepsilon}}:=\int\delta_{\Lambda_{q}'}^{\left[\bar{w},r,q\right]}\dnu_{s,\varepsilon}$
and a second time to $\delta_{\Lambda_{q}'}^{\left[\bar{w},r,q\right]}$,
we get that
\[
H_{\delta_{\Lambda_{q}'}^{\ln\left(q\right)\Delta_{s}}}\left(\pp^{m}\right)\geq\nu\left(\Omega\right)H_{\delta_{\Lambda_{q}'}^{\nu_{s,\varepsilon}}}\left(\pp^{m}\right)\geq\left(1-3\sqrt{\varepsilon}\right)\int H_{\delta_{\Lambda_{q}'}^{\left(\bar{w},r,q\right)}}\left(\pp^{m}\right)\dnu_{s,\varepsilon}\left(\bar{w},r\right).
\]
Using \lemref{main_lemma} we conclude that 
\begin{align*}
\frac{1}{m}H_{\mu}\left(\pp^{m}\right) & \geq\left(1-3\sqrt{\varepsilon}\right)\limsup_{q\to\infty}\left[\left(n-1\right)-\frac{1}{r_{\varepsilon}}\left(\left(n-1\right)-\frac{\ln\left|\Lambda_{q}'\right|}{\ln\left(q\right)}\right)-O\left(2\sqrt{\varepsilon}+\frac{m\ln\left|\pp\right|}{r_{\varepsilon}\ln\left(q\right)}\right)\right]\\
 & \geq\left(1-3\sqrt{\varepsilon}\right)\left[\left(n-1\right)-\frac{\left(n-1\right)-\ell}{r_{\varepsilon}}-O\left(2\sqrt{\varepsilon}\right)\right].
\end{align*}
Assume first that $r\left(\nu_{s}\right)=1-\left(n-1\right)s>0$.
Then $\frac{1}{m}H_{\mu}\left(\pp^{m}\right)\geq\left(1-3\sqrt{\varepsilon}\right)\left[\left(n-1\right)-\frac{\left(n-1\right)-\ell}{1-ds}-O\left(2\sqrt{\varepsilon}\right)\right]\to\left(n-1\right)\left(1-\frac{1-\ell/\left(n-1\right)}{1-\left(n-1\right)s}\right)$
as $\varepsilon\to0$.

If $\ell=\left(n-1\right)$, then regardless of whether $r\left(\nu_{s}\right)>0$,
we have that $\frac{1}{m}H_{\mu}\left(\pp^{m}\right)\geq\left(1-3\sqrt{\varepsilon}\right)\left[\left(n-1\right)-O\left(2\sqrt{\varepsilon}\right)\right]\to\left(n-1\right)$
as $\varepsilon\to0$, thus completing the proof.
\end{proof}
There are two natural settings in which we can apply \thmref{main_pre_theorem}.
The first one is when $\Lambda_{q}'=\Lambda_{q}$ and then ${\displaystyle \liminf_{q\to\infty}}\frac{\ln\left|\Lambda_{q}'\right|}{\ln\left(q\right)}=n-1$
so we only need to worry about the escape of mass condition. We shall
deal with this setting in the next section. 

The second setting is when the $\delta_{\Lambda_{q}'}^{\ln\left(q\right)\Delta_{s}},\int\delta_{\Lambda_{q}'a_{q}\left(\bar{w}\right)}\dnu_{s}$
are supported in some fixed compact set $K$, hence condition \enuref{no_escape}
is automatically satisfied, and any partial weak limit $\mu$ is also
compactly supported probability measure. Thus, if $\Lambda_{q}'$
is big enough, then $\mu$ has positive entropy. The classification
of such measures for $n\geq3$ was done in \cite{einsiedler_invariant_2006}
by Einsiedler, Katok and Lindenstrauss.
\begin{thm}
\label{thm:positive_classification}(see \cite{einsiedler_invariant_2006})
Let $\mu$ be an $A$-invariant and ergodic probability measure on
$X$ with positive entropy with respect to some nontrivial element
in $A$. Then $\mu$ is an algebraic measure and is not compactly
supported.
\end{thm}
 From this we can prove \thmref{bounded_orbits_uniform}.
\begin{proof}[Proof of \thmref{bounded_orbits_uniform}]
Fix a compact set $K\subseteq X_{n}$. For $s>0$ we have that $\Lambda_{q,K}=\left\{ x\in\Lambda_{q}\;\mid\;x\cdot\left(\ln\left(q\right)\Delta\right)\subseteq K\right\} \subseteq\Lambda_{q,K,s}:=\left\{ x\in\Lambda_{q}\;\mid\;x\cdot\left(\ln\left(q\right)\Delta_{s}\right)\subseteq K\right\} $.
By \thmref{main_pre_theorem} we get that if $\limfi i{\infty}\frac{\ln\left|\Lambda_{q_{i},K,s}\right|}{\ln\left(q_{i}\right)}=\ell>0$
for any subsequence, then any partial limit $\mu$ of $\delta_{\Lambda_{q_{i}}'}^{\ln\left(q_{i}\right)\Delta_{s}}$
is an $A$-invariant probability measure with $h_{\mu}\left(T\right)\geq\left(n-1\right)\left(1-\frac{1-\ell/\left(n-1\right)}{1-\left(n-1\right)s}\right)$.
In particular, choosing $s$ small enough we get a positive entropy,
hence $\mu$ is not bounded in contradiction to \thmref{positive_classification}.
We conclude that $\limfi q{\infty}\frac{\ln\left|\Lambda_{q_{i},K,s}\right|}{\ln\left(q\right)}=0$,
hence $\left|\Lambda_{q,K}\right|=O\left(q^{\varepsilon}\right)$
for any $\varepsilon>0$.
\end{proof}

\subsection{\label{subsec:No-escape}No escape of mass and ergodicity}

In this section we show that for any $\left(\bar{w},r\right)\in\Delta$,
the measures $\delta_{\Lambda_{q}}^{\left[\bar{w},r,q\right]},\delta_{\Lambda_{q}a_{q}\left(\bar{w}\right)}$
do not exhibit escape of mass (i.e. when averaging over all the elements
in $\Lambda_{q}$). Hence condition (\enuref{no_escape}) in \thmref{main_pre_theorem}
is satisfied. Since $\frac{\ln\left|\Lambda_{q}\right|}{\ln\left(q\right)}=\left(n-1\right)\frac{\ln\left(\varphi\left(q\right)\right)}{\ln\left(q\right)}\to\left(n-1\right)$,
we obtain that any partial weak limit of $\delta_{\Lambda_{q}}^{\Delta_{s}}$
has maximal entropy and therefore must be the Haar measure. As we
shall see, the argument for the lack of escape of mass boils down
to equidistribution of $\left(\nicefrac{\ZZ}{q\ZZ}\right)^{\times}$
in $\nicefrac{\ZZ}{q\ZZ}$, and we begin with this result. The results
in this section were proved in \cite{david_equidistribution_2017}
for $n=2$, and the proofs here are straight forward generalizations
to higher dimension.
\begin{defn}
For $N\in\NN$ we set $\omega\left(N\right)$ to be the number of
distinct prime factors of $N$.
\end{defn}
\begin{lem}
Let $N\in\NN$ and $0\leq\alpha\leq1$. Then $\left|\left|\left\{ 1\leq m\leq\alpha N\;:\;\left(m,N\right)=1\right\} \right|-\alpha\varphi\left(N\right)\right|\leq2^{\omega\left(N\right)}$.
\end{lem}
\begin{proof}
For $P\in\NN$ set $U_{P}=P\ZZ\cap\left[1,...,\alpha N\right]$, so
that $\left|U_{p}\right|=\flr{\frac{\alpha N}{P}}$. We want to find
$U_{1}\backslash\bigcup_{p}U_{p}$ where $p$ runs over the prime
divisors of $N$. Letting $\mu\left(P\right)$ be the M{\"o}bius
function we obtain that
\[
\left|U_{1}\backslash\bigcup_{p}U_{p}\right|=\sum_{P\mid N}\mu\left(P\right)\left|U_{P}\right|=\alpha N\sum_{P\mid N}\frac{\mu\left(P\right)}{P}+\sum_{P\mid N}\mu\left(P\right)\left[\left|U_{P}\right|-\frac{\alpha N}{P}\right].
\]
The lemma is proved by noting that $N\sum_{P\mid N}\frac{\mu\left(P\right)}{P}=\varphi\left(N\right)$
and that 
\[
\left|\sum_{P\mid N}\mu\left(P\right)\left[\left|U_{P}\right|-\alpha\frac{N}{P}\right]\right|\leq\sum_{P\mid N}\left|\mu\left(P\right)\right|\leq2^{\omega\left(N\right)}.
\]

\end{proof}
The next lemma shows that sequence of the form $\delta_{\Lambda_{q}a\left(\bar{t}\right)}$
do not exhibit escape of mass under suitable conditions on $\bar{t}$.
\begin{lem}[No escape of mass]
\label{lem:mass_upper_bound} Fix some $q\in\NN,\;M>1$ and $\bar{t}\in\ln\left(q\right)\Delta$
where $\ln\left(q\right)\geq2\omega\left(q\right)+{\displaystyle \max_{i\geq2}}t_{i}-t_{1}$.
Then
\[
\left|\left\{ \bar{p}\in\Lambda_{q}\;\mid\;\Gamma u_{\bar{p}/q}a\left(\bar{t}\right)\in X^{\geq M}\right\} \right|\leq\frac{\left(4\varphi\left(q\right)\right)^{n-1}}{M^{n}}=\frac{4^{n-1}}{M^{n}}\left|\Lambda_{q}\right|,
\]
or equivalently $\left(\delta_{\Lambda_{q}a\left(\bar{t}\right)}\right)\left(X^{\geq M}\right)\leq\frac{4^{n-1}}{M^{n}}$.
\end{lem}
\begin{proof}
We say that $\bar{p}\in\Lambda_{q}$ is bad if $\Gamma u_{\bar{p}/q}a\left(\bar{t}\right)\in X^{\geq M}$,
i.e. there exists $\bar{0}\neq\left(m,\bar{k}\right)\in\ZZ^{n}$ such
that $\norm{\left(m,\bar{k}\right)u_{\bar{p}/q}a\left(\bar{t}\right)}_{\infty}\leq\frac{1}{M}$
or equivalently $\left|m\right|\leq\frac{1}{Me^{t_{1}}}$ and $\left|mp_{i}+qk_{i}\right|\leq\frac{q}{Me^{t_{i}}}$
for $i=2,...,n$. Clearly, we may assume that $m\geq0$ and actually
even $m\geq1$, since otherwise for each $i\geq2$ we have $1>\frac{1}{M}\geq\left|\left(m\frac{p_{i}}{q}+k_{i}\right)e^{t_{i}}\right|=\left|k_{i}\right|e^{t_{i}}\geq\left|k_{i}\right|$
and therefore $k_{i}=0$ (because $t_{i}\geq0$ for $i\geq2$) - contradiction.

We say that $\bar{p}$ is $m$-bad if there exist some $\bar{n}$
such that $\left|mp_{i}+qk_{i}\right|\leq\frac{q}{Me^{t_{i}}}$ for
$i=2,...,n$.

Fixing $1\leq m\leq\frac{1}{Me^{t_{1}}}$, we count the number of
$m$-bad $\bar{p}$ for each such $m$. Letting $d_{m}=gcd\left(q,m\right)$
and writing $q=\tilde{q}d_{m},\;m=\tilde{m}d_{m}$, we get that 
\begin{equation}
\left|\tilde{q}k_{i}+\tilde{m}p_{i}\right|\leq\frac{\tilde{q}}{e^{t_{i}}M}.\label{eq:2313}
\end{equation}
We consider \eqref{2313} as a counting problem over $\nicefrac{\ZZ}{\tilde{q}\ZZ}$.
Note first that $\bar{t}\in\ln\left(q\right)\Delta$  implies that
$t_{1}\geq-\ln\left(q\right)\frac{n-1}{n}$, so that $d_{m}\leq m<q^{\frac{n-1}{n}}$
and therefore $\tilde{q}>q^{\frac{1}{n}}>1$. Since $\left(q,p\right)=1$
we also have that $\left(\tilde{q},p\right)=1$ and hence $\left(\tilde{q},\tilde{m}p\right)=1$.
Consider the map $\pi:\left(\nicefrac{\ZZ}{q\ZZ}\right)^{\times}\to\left(\nicefrac{\ZZ}{\tilde{q}\ZZ}\right)^{\times}$
and let $A_{i}=\left\{ \left[a\right]\in\left(\nicefrac{\ZZ}{\tilde{q}\ZZ}\right)^{\times}\;\mid\;\left|a\right|\leq\frac{\tilde{q}}{e^{t_{i}}M}\right\} $.
For this fixed $m$, the $m$-bad $\bar{p}$'s are exactly $\prod_{1}^{k}\pi^{-1}\left(\tilde{m}^{-1}A_{i}\right)$,
hence there are at most $\prod_{2}^{n}\left(\left|A_{i}\right|\cdot\left|\ker\left(\pi\right)\right|\right)$
such $\bar{p}$. Since $\pi$ is surjective we obtain that $\left|\ker\left(\pi\right)\right|=\frac{\varphi\left(q\right)}{\varphi\left(\tilde{q}\right)}$
and by the previous lemma we get that $\left|A_{i}\right|\leq2\left(\frac{1}{e^{t_{i}}M}\varphi\left(\tilde{q}\right)+2^{\omega\left(\tilde{q}\right)}\right)$.

We claim that $2^{\omega\left(\tilde{q}\right)}\leq\frac{1}{e^{t_{i}}M}\varphi\left(\tilde{q}\right)$.
Assuming this claim, the total number of $m$-bad $p$'s is at most
$\prod_{2}^{n}\left|A_{i}\right|\cdot\left|\ker\left(\pi\right)\right|\leq\prod_{2}^{n}\left(\frac{4}{e^{t_{i}}M}\varphi\left(q\right)\right)=e^{t_{1}}\left(\frac{4\varphi\left(q\right)}{M}\right)^{n-1}$.
Since there are $\flr{\frac{1}{Me^{t_{1}}}}$ such $m$, a union bound
shows that the number of bad $p$ is at most $\frac{\left(4\varphi\left(q\right)\right)^{n-1}}{M^{n}}$.
Thus, to complete the proof we need only to show that $2^{\omega\left(\tilde{q}\right)}\leq\frac{1}{e^{t_{i}}M}\varphi\left(\tilde{q}\right)$
for $i\geq2$ and $\bar{t}\in\ln\left(q\right)\Delta$ such that $\ln\left(q\right)\geq2\omega\left(q\right)+{\displaystyle \max_{i\geq2}}t_{i}-t_{1}$.

For any integer with prime decomposition $N=\prod r_{i}^{e_{i}}$
we have that
\[
\varphi\left(N\right)=\prod r_{i}^{e_{i}}\left(1-\frac{1}{r_{i}}\right)\geq\frac{1}{e^{\omega\left(N\right)}}\prod r_{i}^{e_{i}}=\frac{N}{e^{\omega\left(N\right)}}.
\]
The claim now follows from
\begin{eqnarray*}
\frac{2^{\omega\left(\tilde{q}\right)}}{\varphi\left(\tilde{q}\right)} & \leq & \frac{e^{\omega\left(\tilde{q}\right)}}{e^{-\omega\left(\tilde{q}\right)}\frac{q}{d_{m}}}=d_{m}e^{2\omega\left(\tilde{q}\right)-\ln\left(q\right)}\leq\frac{1}{Me^{t_{1}}}\cdot e^{2\omega\left(q\right)-\ln\left(q\right)}\leq\frac{1}{Me^{t_{i}}}.
\end{eqnarray*}
\end{proof}
The condition $\ln\left(q\right)\geq2\omega\left(q\right)+{\displaystyle \max_{i\geq2}}t_{i}-t_{1}$
in the lemma means that $\bar{t}$ is a bit far from the faces of
$\Delta$ which do not contain $\bar{0}$ (these are the dashed lines
in \figref{example2}). Since $\omega\left(q\right)=o\left(\ln\left(q\right)\right)$,
this is a very weak condition which we use in the next lemma.
\begin{lem}
\label{lem:no_escape}For any $\left(\bar{w},r\right)\in E\left(\Delta\right)$
and $q\in\NN$ large enough so that $\frac{1}{r}\frac{2\omega\left(q\right)}{\ln\left(q\right)}\leq\min\left\{ 1,\frac{4^{n-1}}{M^{n}}\right\} $
we have that $\delta_{\Lambda_{q}}^{\left[\bar{w},r,q\right]}\left(X^{\geq M}\right),\delta_{\Lambda_{q}a_{q}\left(\bar{w}\right)}\left(X^{\geq M}\right)\leq\left(\frac{4}{M}\right)^{n}$.
\end{lem}
\begin{proof}
For $s\in\RR$ we denote $\bar{w}^{\left(s\right)}=\bar{w}+s\overrightarrow{v}$
and note that for $r,s\in\RR$ we have that 
\[
\max_{i\geq2}w_{i}^{\left(s\right)}-w_{1}^{\left(s\right)}=\left(s-r\right)+\max_{i\geq2}w_{i}^{\left(r\right)}-w_{1}^{\left(r\right)}.
\]
The assumption that $\bar{w}^{\left(r\right)}\in\Delta$ implies that
${\displaystyle \max_{i\geq2}w_{i}^{\left(r\right)}-w_{1}^{\left(r\right)}}\leq1$.
Applying \lemref{mass_upper_bound} we conclude that if $r-s\geq\frac{2\omega\left(q\right)}{\ln\left(q\right)}$,
then $\delta_{\Lambda_{q}a_{q}\left(\bar{w}^{\left(s\right)}\right)}\left(X^{\geq M}\right)\leq\frac{4^{n-1}}{M^{n}}$.
By assumption, this is true for $s=0$ and therefore $\delta_{\Lambda_{q}a_{q}\left(\bar{w}\right)}\left(X^{\geq M}\right)\leq\frac{4^{n-1}}{M^{n}}$.
The proof is completed by noting that
\begin{align*}
\delta_{\Lambda_{q}}^{\left[\bar{w},r,q\right]}\left(X^{\geq M}\right) & =\frac{1}{r}\int_{0}^{r}\delta_{\Lambda_{q}a_{q}\left(\bar{w}^{\left(s\right)}\right)}\left(X^{\geq M}\right)\ds\leq\frac{1}{r}\left(r\frac{4^{n-1}}{M^{n}}+\frac{2\omega\left(q\right)}{\ln\left(q\right)}\right)\leq\left(\frac{4}{M}\right)^{n}.
\end{align*}
\end{proof}

We can now use this lack of escape of mass and \thmref{main_pre_theorem}
in order to prove \thmref{main_theorem}.
\begin{proof}[Proof of \thmref{main_theorem}]
 We first prove $(3)$.

By \lemref{no_escape} and \thmref{main_pre_theorem} we get that
$\delta_{\Lambda_{q}}^{\ln\left(q\right)\Delta}\wstar\mu_{Haar}$,
hence part (3) follows by applying \corref{fund_domain}.

$(3)\Rightarrow(2)$: Suppose that $\left|\Lambda_{q}'\right|\geq\alpha\left|\Lambda_{q}\right|$
for some $\alpha>0$ and write 
\[
\delta_{\Lambda_{q}}^{\ln\left(q\right)\Delta_{full}}=\frac{\left|\Lambda_{q}'\right|}{\left|\Lambda_{q}\right|}\delta_{\Lambda_{q}'}^{\ln\left(q\right)\Delta_{full}}+\frac{\left|\Lambda_{q}\backslash\Lambda_{q}'\right|}{\left|\Lambda_{q}\right|}\delta_{\Lambda_{q}\backslash\Lambda_{q}'}^{\ln\left(q\right)\Delta_{full}}.
\]
Suppose that $\mu_{1}$ is a partial weak limit of $\delta_{\Lambda_{q}'}^{\ln\left(q\right)\Delta_{full}}$,
and note in particular that it is $T$-invariant. By restricting to
a subsequence we may assume that all the terms above converge and
by taking the limit and applying part (3) we obtain that $\mu_{Haar}=\alpha'\mu_{1}+\left(1-\alpha'\right)\mu_{2}$
where $\alpha'\geq\alpha>0$. If $\alpha'=1$, then we are done. Otherwise
we must have that $\mu_{2}$ is also $T$-invariant and both $\mu_{1},\mu_{2}$
are probability measures. Since $\mu_{Haar}$ is $T$-ergodic, it
is an extreme point in the set of $T$-invariant probability measures,
implying that $\mu_{1}=\mu_{2}=\mu_{Haar}$. It follows that $\delta_{\Lambda_{q}'}^{\ln\left(q\right)\Delta_{full}}$
converges to the Haar measure as well.

$(2)\Rightarrow(1)$: Let $\left\{ f_{i}\right\} _{1}^{\infty}$ be
a dense sequence of compactly supported continuous function in $C_{c}\left(X\right)$.
For each $m,q\in\NN$ define
\[
\Lambda_{q,m}=\left\{ x\in\Lambda_{q}\;\mid\;\max_{1\leq i\leq m}\left|\left(\delta_{x}^{\ln\left(q\right)\Delta_{full}}-\mu_{Haar}\right)\left(f_{i}\right)\right|<\frac{1}{m}\right\} .
\]
We claim that $\limfi q{\infty}\frac{\left|\Lambda_{q,m}\right|}{\left|\Lambda_{q}\right|}=1$
for any fixed $n$. Otherwise, find some $1\leq i\leq m$, $\epsilon\in\left\{ \pm1\right\} $
and $\alpha>0$ such that 
\[
V_{q}:=\left\{ x\in\Lambda_{q}\;\mid\;\eps\left(\delta_{x}^{\ln\left(q\right)\Delta_{full}}-\mu_{Haar}\right)(f_{i})\geq\frac{1}{m}\right\} 
\]
satisfy $\frac{\left|V_{q_{j}}\right|}{\left|\Lambda_{q}\right|}\geq\alpha$
for some subsequence $q_{j}$. By part $(2)$ we get that $\delta_{V_{q}}^{\ln\left(q\right)\Delta_{full}}\to\mu_{Haar}$
while $\delta_{V_{q}}^{\ln\left(q\right)\Delta_{full}}\left(f_{i}\right)\not\to\mu_{Haar}\left(f_{i}\right)$
- contradiction. Define $m\left(q\right)=\max\left\{ m\;\mid\;\frac{\left|\Lambda_{q,m}\right|}{\left|\Lambda_{q}\right|}\geq1-\frac{1}{m}\right\} $
and set $\Lambda_{q}'=\Lambda_{q,m\left(q\right)}$. By the claim
that we just proved we get that $m\left(q\right)\overset{q\to\infty}{\longrightarrow}\infty$,
and therefore $\frac{\left|\Lambda_{q}'\right|}{\left|\Lambda_{q}\right|}\geq1-\frac{1}{m\left(q\right)}\to1$.
Furthermore, for each $i\leq m$ we have that $\left|\left(\delta_{x}^{\ln\left(q\right)\Delta_{full}}-\mu_{Haar}\right)\left(f_{i}\right)\right|<\frac{1}{m}$
whenever $x\in\Lambda_{q}'$ and $m\left(q\right)\geq m$. We conclude
that $\delta_{x^{\left(q\right)}}^{\ln\left(q\right)\Delta_{full}}\left(f_{i}\right)\to\mu_{Haar}\left(f_{i}\right)$
for any choice of $x^{\left(q\right)}\in\Lambda_{q}'$, and since
$\left\{ f_{i}\right\} _{1}^{\infty}$ is dense we have that $\delta_{x^{\left(q\right)}}^{\ln\left(q\right)\Delta_{full}}\wstar\mu_{Haar}$
and we are done.
\end{proof}

\appendix

\section{\label{app:bowen}Proof of \lemref{bowen_control}}

Before proving \lemref{bowen_control}, we give a few result about
small balls in $X_{n}$ which we need in the proof. Recall (from \defref{small_balls})
that for $\eta>0,N\geq0$ we set 
\begin{align*}
V_{\eta,N} & =\left\{ \sum_{i,j=1}^{n}a_{i,j}E_{i,j}\;\mid\;\left|a_{i,j}\right|\leq\begin{cases}
\eta e^{-N} & i=1,\;j\geq2\\
\eta & i\geq2\;or\;i=j=1
\end{cases}\quad\quad\right\} ,\\
B_{\eta,N} & =\left(I+V_{\eta,N}\right)\cap\SL_{n}\left(\RR\right),
\end{align*}
and denote $V_{\eta}=V_{\eta,0}$ and $B_{\eta}=B_{\eta,0}$.

It is easy to check that $V_{\eta,N}\cdot V_{\eta',N}\subseteq V_{\eta\eta'n,N}$
and $V_{\eta,N}+V_{\eta',N}\subseteq V_{\eta+\eta',N}$. In particular
we get the following two simple results.
\begin{lem}
\label{lem:inverse_ball}Let $I+W\in B_{\eta,N}$ where $\eta<\frac{1}{2n}$.
Then $\left(I+W\right)^{-1}\in I-W+V_{\eta^{2}n,N}$ and in particular
$\left(B_{\eta,N}\right)^{-1}\subseteq B_{\frac{3}{2}\eta,N}$.
\end{lem}
\begin{proof}
Given any $W\in V_{\eta,N}$ we have that $W^{k}\in V_{\eta\left(\eta n\right)^{k},N}$
and hence $\sum_{2}^{\infty}\left(-W\right)^{k}\in\sum_{1}^{\infty}V_{\eta\left(\eta n\right)\frac{1}{2^{k}},N}=V_{\eta\left(\eta n\right),N}$
is well defined. It then follows that $\left(I+W\right)^{-1}=I-W+\sum_{2}^{\infty}\left(-W\right)^{k}\in I-W+V_{\eta\left(\eta n\right),N}$.
\end{proof}
\begin{lem}
\label{lem:change_center}Let $g,h\in\Gamma\backslash G$ such that
$h\in gB_{\eta,N}$ for $\eta<\frac{1}{2n}$. Then $gB_{\eta,N}\subseteq hB_{4\eta,N}$.
\end{lem}
\begin{proof}
Write $h=g\tilde{h}$ for some $\tilde{h}\in B_{\eta,N}$. Then $gB_{\eta,N}=h\tilde{h}^{-1}B_{\eta,N}\subseteq hB_{\eta,N}^{-1}B_{\eta,N}\subseteq yB_{\frac{3}{2}\eta,N}B_{\eta,N}$.
The claim now follows from the fact that $V_{\frac{3}{2}\eta,N}+V_{\eta,N}+V_{\frac{3}{2}\eta,N}V_{\eta,N}\subseteq V_{4\eta,N}$.
\end{proof}
The following lemma shows that we can cover an $\left(\eta,N\right)$
ball by $\left(\frac{\eta}{R},N\right)$ balls where the number of
small balls needed is bounded as a function of $R$.
\begin{lem}
\label{lem:cover_constant}Let $R\in\RR_{>1}$ and $0<\eta<\frac{1}{6nR}$.
There exists less than $\left(8R\right)^{n^{2}}$ elements $\gamma_{i}\in B_{\eta,N}$
such that $B_{\eta,N}\subseteq\bigcup\gamma_{i}B_{\frac{\eta}{R},N}$.
If $S\subseteq B_{\eta,N}$ is given, then we can choose $\gamma_{i}\in S$
such that $S\subseteq\bigcup\gamma_{i}B_{\frac{\eta}{R},N}$.
\end{lem}
\begin{proof}
The cover argument is obvious if we work in the additive world, i.e.
in $V_{\bar{r}}$ instead, so we start there and write $V_{\eta,N}=\bigcup_{1}^{K}\left(v_{i}+V_{\frac{\eta}{4R},N}\right)$
where $K=\left(4\left\lceil R\right\rceil \right)^{n^{2}}\leq\left(8R\right)^{n^{2}}$.
Suppose that $W_{1},W_{2}\in v_{i}+V_{\frac{\eta}{4R},N}$, and we
wish to show that $\left(I+W_{1}\right)^{-1}\left(I+W_{2}\right)$
are in a small ball $B_{\frac{\eta}{R},N}$.

If $W_{1},W_{2}\in v_{i}+V_{\frac{\eta}{4R},N}$, then setting $W_{1}'=\sum_{2}^{\infty}\left(-W\right)^{k}$,
which is in $V_{\eta^{2}n,N}$ by \lemref{inverse_ball}, we get that
\begin{align*}
\left(I+W_{1}\right)^{-1}\left(I+W_{2}\right) & =\left(I-W_{1}+W_{1}'\right)\left(I+W_{2}\right)=I+\left(W_{2}-W_{1}\right)-W_{1}W_{2}+W_{1}'+W_{1}'W_{2}\\
 & \in I+V_{\frac{\eta}{2R},N}+V_{\eta^{2}n,N}+V_{\eta^{2}n,N}+V_{\eta^{3}n^{2},N}\subseteq I+V_{\frac{\eta}{2R},N}+V_{3\eta^{2}n,N}\subseteq I+V_{\frac{\eta}{R},N}.
\end{align*}

For each $1\leq i\leq K$ choose $\gamma_{i}\in\SL_{n}\left(\RR\right)\cap\left(I+v_{i}+V_{\frac{\eta}{4R},N}\right)$
if this set is not empty, so by the argument above, any $I+W\in\SL_{n}\left(\RR\right)\cap\left(I+V_{\frac{\eta}{R},N}\right)$
satisfies $\gamma_{i}^{-1}\left(I+W\right)\in B_{\frac{\eta}{R},N}$,
and we conclude that $B_{\eta,R}\subseteq\bigcup_{1}^{K}\gamma_{i}B_{\frac{\eta}{R},N}$.

If $S\subseteq B_{\eta,N}$ is given, then choosing $\gamma_{i}\in S\cap\left(I+v_{i}+V_{\frac{\eta}{4r},N}\right)$
will satisfy the requirements of the lemma.
\end{proof}
We now use these lemmas in order to prove \lemref{bowen_control}.
\begin{proof}[Proof of \lemref{bowen_control}]
Choose $0<\eta_{0}\left(M,n\right)<\frac{1}{6n}$ to be small enough
so that the map $g\mapsto xg$ from $B_{\eta_{0}e}\to\Gamma\backslash G$
is injective for all $x\in X^{\leq M}$ and let $\pp=\left\{ P_{0},...,P_{\ell}\right\} $
be an $\left(M,\eta\right)$ partition, $\eta<\eta_{0}\left(M\right)$. 

Consider the function $f\left(x\right)=\frac{1}{N}\sum_{0}^{N-1}1_{X^{>M}}\left(T^{i}x\right)$
and note that this function is constant on each $P\in\pp_{N}$.

Setting $X'=X^{\leq M}\cap\left\{ x:f\left(x\right)\leq\kappa\right\} $,
we obtain that
\begin{align*}
1 & \leq\mu\left(X^{>M}\right)+\mu\left(\left\{ f\left(x\right)>\kappa\right\} \right)+\mu\left(X'\right)\leq\mu\left(X^{>M}\right)+\kappa^{-1}\int f\left(x\right)\mathrm{d\mu}+\mu\left(X'\right)\\
 & =\mu\left(X^{>M}\right)+\mu^{N}\left(X^{>M}\right)\kappa^{-1}+\mu\left(X'\right)
\end{align*}
thus proving part $(3)$ in the theorem.

Suppose that $S\in\pp_{N}$, $S\subseteq X'$ and let $V_{m}=\left|\left\{ 0\leq i\leq m\;\mid\;T^{i}\left(S\right)\subseteq X^{>M}\right\} \right|$.
Let $C$ be the constant from \lemref{cover_constant} For $R=e$.
We claim that $S\subseteq\bigcup_{1}^{C^{\left|V_{m}\right|}}h_{i,m}B_{\eta,m}$
where $h_{i,m}\in S$ for any $0\leq m\leq N$, and the lemma will
follow by taking $m=N-1$. For $m=0$, let $h\in S\subseteq P_{i}\subseteq g_{i}B_{\frac{\eta}{4}}$
for some $i\geq1$, so by \lemref{change_center} we obtain that $S\subseteq hB_{\eta}$,
thus proving the case for $m=0$.

Assume that $S\subseteq\bigcup_{i=1}^{C^{\left|V_{m}\right|}}h_{i,m}B_{\eta,m},\ h_{i,m}\in S$
for $m<N-1$ and we prove for $m+1$. We set $a=a\left(\frac{1}{n}\left(1-n,1,...,1\right)\right)\in A$.
\begin{itemize}
\item Suppose first that $T^{m+1}S\subseteq X^{\leq M}$ so that $T^{m+1}S\subseteq P_{j}\subseteq g_{j}B_{\frac{\eta}{4}}$
for some $j\geq1$. This case will be complete if we can show that
$S\cap h_{i,m}B_{\eta,m}=S\cap h_{i,m}B_{\eta,m+1}$ for every $i$.
Indeed, \lemref{change_center} implies that $T^{m+1}S\subseteq h_{i,m}a^{m+1}B_{\eta}$,
so if $h_{i,m}g\in S$ with $g\in B_{\eta,m}$, then
\[
\left[h_{i,m}a^{\left(m+1\right)}\right]a^{-\left(m+1\right)}ga^{\left(m+1\right)}=h_{i,m}ga^{\left(m+1\right)}\in T^{m+1}S\subseteq y_{i}a^{\left(m+1\right)}B_{\eta}.
\]
By the assumption on the injectivity radius and since $a^{-\left(m+1\right)}ga^{\left(m+1\right)}\in B_{\eta e}$
, we conclude that $g\in B_{\eta,m}\cap a^{\left(m+1\right)}B_{\eta}a^{-\left(m+1\right)}=B_{\eta,m+1}$
which is what we wanted to show.
\item Suppose now that $T^{n+1}S\subseteq X^{>M}$. By \lemref{cover_constant},
for each $i$ we have that $S\cap h_{i,m}B_{\eta,m}\subseteq\bigcup_{j=1}^{C}\tilde{h}_{i,m}^{(j)}B_{\frac{\eta}{4e},m}$.
Choose $h_{i,m}^{(j)}\in S\cap\tilde{h}_{i,m}^{(j)}B_{\frac{\eta}{4e},m}$
if this set is not empty and otherwise choose $h_{i,m}^{(j)}\in S$
arbitrarily. By \lemref{change_center} it follows that $\bigcup_{j=1}^{C}\tilde{h}_{i,m}^{(j)}B_{\frac{\eta}{4e},m}\subseteq\bigcup_{j=1}^{C}h_{i,m}^{(j)}B_{\frac{\eta}{e},m}\subseteq\bigcup_{j=1}^{C}h_{i,m}^{(j)}B_{\eta,m+1}$,
which completes the proof.
\end{itemize}
\end{proof}
\bibliographystyle{plain}
\bibliography{../bib}

\end{document}